\documentclass[11pt,a4paper]{article}
\usepackage[T1]{fontenc}
\usepackage[utf8]{inputenc}
\usepackage{lmodern}
\usepackage{textcomp}
\usepackage{microtype}
\usepackage[font=small,labelfont=bf]{caption}
\usepackage{titlesec}
\usepackage{tocloft}
\usepackage{amsmath,amssymb}
\usepackage{amsfonts}
\usepackage{graphicx}
\usepackage{amsthm}
\usepackage{yfonts}
\usepackage{tikz-cd}
\usepackage{amssymb}
\usepackage{mathrsfs}
\usepackage[parfill]{parskip}
\usepackage{hyperref}
\usepackage[nameinlink]{cleveref}
\usepackage{MnSymbol }
\usepackage{verbatim}
\usepackage{array}
\newtheorem{theorem}{Theorem}

\newtheorem{corollary}{Corollary}[theorem]
\newtheorem{lemma}[theorem]{Lemma}
\newtheorem{prop}[theorem]{Proposition}
\theoremstyle{definition}
\newtheorem{rem}[theorem]{Remark}

 \titleformat{\subsection}[runin]{\normalsize\bf}{\thesubsection.}{}{}[]
 \titleformat{\subsubsection}[runin]{\normalsize\bf}{\thesubsubsection.}{}{}[]

\newcommand{\fakesubsection}[1]{
  \par\refstepcounter{subsection}
  \subsectionmark{#1}
  \textbf{\arabic{section}.\arabic{subsection}}
}
\newcommand{\fakesubsubsection}[1]{
  \par\refstepcounter{subsubsection}
  \subsubsectionmark{#1}
  \textbf{\arabic{section}.\arabic{subsection}.\arabic{subsubsection}}
}
\newcommand\restr[2]{{
  \left.\kern-\nulldelimiterspace 
  #1
  \vphantom{\big|}
  \right|_{#2}
  }}
\newcommand{\gl}{\mathrm{GL}}
\newcommand{\s}{\mathrm{Sp}}
\newcommand{\C}{\mathbb{C}}
\newcommand{\id}{\mathrm{Ind}}

\newcommand{\bs}{\backslash}
\newcommand{\ho}{\mathrm{Hom}}
\newcommand{\sh}{\mathrm{Sh}}
\newcommand{\scc}{\mathrm{Sec}}
\newcommand{\scf}{\mathscr{F}}
\newcommand{\rs}{\mathrm{Res}}
\newcommand{\rep}{\mathrm{Rep}}

\newcommand{\mvw}{^{\mathrm{MVW}}}
\newcommand{\ra}{\rightarrow}
\newcommand{\hra}{\hookrightarrow}
\newcommand{\sra}{\twoheadrightarrow}
\newcommand{\ir}{\mathrm{Irr}}
\newcommand{\ww}{\mathbb{W}}
\newcommand{\m}{\mathrm{Mp}}
\newcommand{\oo}{\mathrm{O}}
\newcommand{\uu}{\mathrm{U}}
\newcommand{\wt}{\widetilde}

\newcommand{\Z}{\mathrm{Z}}
\renewcommand{\L}{\mathrm{L}}
\newcommand{\NN}{\mathbb{N}}
\newcommand{\ZZ}{\mathbb{Z}}
\newcommand{\jac}{\mathrm{Jac}_{\rho_t}}
\newcommand{\iso}{\xrightarrow{\sim}}
\newcommand{\ain}[3]{#1\in\{#2,\ldots,#3\}}
\title{Proof of a conjecture of Kudla and Rallis on quotients of degenerate principal series}
\author{Johannes Droschl}
\date{}
\begin{document}
\maketitle
\setlength\parindent{0pt}
\begin{abstract}
\noindent
    In this paper we prove a conjecture of Kudla and Rallis, see \cite[Conjecture V.3.2]{kudla1996notes}. Let $\chi$ be a unitary character, $s\in \C$ and $W$ a symplectic vector space over a non-archimedean field with symmetry group $G(W)$. Denote by $I(\chi,s)$ the degenerate principal series representation of $G(W\oplus W)$.  Pulling back $I(\chi,s)$ along the natural embedding $G(W)\times G(W)\hra G(W\oplus W)$ gives a representation $I_{W,W}(\chi,s)$ of $G(W)\times G(W)$. Let $\pi$ be an irreducible smooth complex representation of $G(W)$. We then prove \[\dim _\C\ho_{G(W)\times G(W)}(I_{W,W}(\chi,s),\pi\otimes \pi^\lor)=1.\]
We also give analogous statements for $W$ orthogonal or unitary. This gives in particular a new proof of the conservation relation of the local theta correspondence for symplectic-orthogonal and unitary dual pairs.
\end{abstract}\textbf{Keywords:} Theta correspondence, Conservation relation, $p$-adic groups

\section{Introduction}
Let $F$ be a non-archimedean local field of characteristic different from $2$, residue characteristic $p$ and $\psi\colon F\rightarrow \C^*$ a non-trivial additive character. Let $E$ be either $F$ or a quadratic extension of $F$ and $\epsilon\in \{\pm1\}$.
We consider a $-\epsilon$-hermitian space $V$ over $E$ of dimension $n$ with discriminant character $\Delta_\mathrm{disc}$, and an $\epsilon$-hermitian space $W$ over $E$ of dimension $m$. Let $\{V_r^+\}$ respectively $\{V_r^-\}$ be the two Witt towers with Hasse-invariant $1$ respectively $-1$, discriminant character $\Delta_\mathrm{disc}$ and same dimensional parity as $V$.
Define $\mathbb{W}\coloneq V\otimes W$ together with its naturally induced symplectic form and consider the metaplectic group $\mathrm{Mp}(\ww)$ of $\ww$.
Let $G(W)$ be the metaplectic group $\mathrm{Mp}(W)$ if $W$ is symplectic and $n$ is odd and otherwise the symmetry group of $W$. Similarly, let $G'(V)$ be the metaplectic group $\m(V)$ if $V$ is symplectic and $m$ is odd and otherwise the symmetry group of $V$. Then $(G'(V), G(W))$ is a dual pair of $\m(\ww)$.
For the sake of exposition, we will focus in the introduction on the case where $G(W)$ is symplectic and the dimension of $V$ is even, \emph{i.e.} $G'(V)$ is of orthogonal type.

Let $\omega_\psi$ be the Weil representation associated to $\psi$. For $\pi$ an irreducible smooth representation of $G(W)$ let $S[\pi]$ be the largest $\pi$-isotypic quotient of $\restr{\omega_\psi}{{G'(V)}\times {G(W)}}$, where we chose a suitable splitting for $(G'(V), G(W))$.
It is then of the form $\Theta_{V,W,\psi}(\pi)\otimes \pi$, where $\Theta_{V,W,\psi}(\pi)$ is either $0$ or a smooth representation of finite length of $G'(V)$, see \cite{Kudla1986}. We denote its cosocle by $\theta(\pi)$.
The following was then conjectured in \cite{Howe1979seriesAI}, \cite{Howe2} and proven by \cite{Waldspurger} in the case $p\neq 2$. In \cite{GanTakeda} a new proof without the assumption on $p$ was given and in \cite{Gan2017} the remaining case of quaternionic dual pairs was covered.
\begin{theorem}[Howe Duality Conjecture for Type I, {\cite{Waldspurger}, \cite{GanTakeda}, \cite{Gan2017}}]
Let $\pi,\pi'$ be irreducible smooth representations of ${G(W)}$. Then the following holds.
\begin{enumerate}
    \item The representation $\theta(\pi)$ is either irreducible or $0$.
    \item If $\theta(\pi)\cong \theta(\pi')\neq 0$, then $\pi\cong \pi'$.
\end{enumerate}
\end{theorem}
We denote by $m_\chi^\pm(\pi)$ the first occurrence of an irreducible smooth representation $\pi$ of $G(W)$ in the theta correspondence
in the Witt tower $\{V_r^\pm\}$, \textit{i.e.} the smallest $r$ such that $\Theta_{V_r^\pm,W,\psi}(\pi)\neq 0$.
The following conservation relation was conjectured in \cite{KudlaRallis} and proven in \cite{sun}.
\begin{theorem}[Conservation Relation, \cite{sun}]\label{T:conservationrelation}
For any irreducible smooth representation $\pi$ of ${G(W)}$
\[m_\chi^+(\pi)+m_\chi^{-}(\pi)=2m+4.\]
\end{theorem}
To prove this theorem, Kudla and Rallis proposed in \cite{KudlaRallis} the following strategy.
Let $\chi$ be a unitary character of $F^*$ and $s\in \C$ and equip $W\oplus W$ with the given symplectic form $\langle.,.\rangle$ on the first copy of $W$ and $-\langle.,.\rangle$ on the second copy of $W$. Let $I(\chi,s)$ be the degenerate principal series of $G(W\oplus W)$, \textit{i.e.} the parabolically induced representation of the character $\chi\lvert\det \lvert^s$ from the Siegel parabolic in $G(W\oplus W)$.
Consider the natural embedding \[\iota\colon G(W)\times G(W)\rightarrow G(W\oplus W)\] and define the restriction of $I(\chi,s)$ to  $G(W)\times G(W)$ as
\[I_{W,W}(\chi,s)\coloneq \iota^*(I(\chi,s)).\]
In \cite{KudlaRallis} the authors construct a non-zero morphism $I_{W,W}(\chi,s)\rightarrow \pi\otimes \pi^\lor$ for all irreducible representations $\pi$ and conjectured that this morphism is up to a scalar unique. The main result of this paper is the following theorem, which positively answers \cite[Conjecture V.3.2]{kudla1996notes}, \emph{cf.} also \cite[Conjecture 1.2]{KudlaRallis}.
\begin{theorem}\label{T:mainconjecture}
For all irreducible smooth representations $\pi$ of $G(W)$,
\[\dim_\C\ho_{G(W)\times G(W)}(I_{W,W}(\chi,s),\pi\otimes \pi^\lor)= 1.\]
\end{theorem}
In \cite[§4]{KudlaRallis} it was proven that \Cref{T:conservationrelation} holds under the assumption of \Cref{T:mainconjecture}.
In the same paper, see \cite[Theorem 1.1, Lemma 1.4]{KudlaRallis}, \Cref{T:mainconjecture} was verified for $W$ a symplectic vector space and $\pi=1$ the trivial representation or $\pi$ a representation not appearing on the boundary, a notation which we will define in a moment.

 We will now sketch out the proof of \Cref{T:mainconjecture} for general irreducible representations 
below. As we assume $W$ in this introduction to be symplectic, we have $m=2m'$ is even and we can decompose $W$ as $X_{m'}\oplus Y_{m'},$ such that on $X_{m'}$ and $Y_{m'}$ the symplectic form vanishes.
To be more precise, we pick a basis $\{ x_1,\ldots x_{m'}\}$ of $X_{m'}$ and $\{ y_1,\ldots y_{m'}\}$ of $Y_{m'}$
with \[\langle x_i,y_j\rangle =\begin{cases}
    1&\text{ if }i=j,\\
    0&\text{ if }i\neq j.\\
\end{cases}\] and set for $\ain{k}{0}{m'}$,
\[X_k\coloneq \langle x_1,\ldots x_{k}\rangle_E,\, Y_k\coloneq \langle y_1,\ldots y_{k}\rangle_E,\, W_k\coloneq \langle x_{k+1},\ldots x_{m'}, y_{k+1},\ldots, y_{m'}\rangle_E.\] Here we use the notation $\langle S\rangle_E$ for the $E$-sub-vector space spanned by a set $S\subseteq W$. For $P$ a parabolic subgroup of a reductive group $G$ with Levi-decomposition $P=M\ltimes N$, we denote by $r_{P}$ the Jacquet-functor from $G$ to $M$ and by $\id_P^G$ the normalized parabolic induction from $M$ to $G$. We write for $i\in \NN$, $\gl_i\coloneq \gl(F^i)$.

The proof builds on three ingredients, the first being a well-known filtration, see \cite[§ 1]{KudlaRallis}, \begin{equation}\label{E:I1}0=I_{-1}\subseteq  I_0\subseteq \ldots\subseteq I_{m'}=I_{W,W}(\chi,s)\end{equation} of $I_{W,W}(\chi,s)$
with subquotients 
\[\sigma_t\coloneq I_{t-1}\bs I_{t}\cong \id_{P_{(t)}\times P_{(t)} }^{G(W)\times G(W)} ({\underbrace{\chi\lvert\det\lvert}_{\gl_t}}^{s+{{\frac{t}{2}}}}\otimes {\underbrace{\chi\lvert\det\lvert}_{\gl_t}}^{s+{\frac{t}{2}}}\otimes S(G(W_t))).\]
Here, $P_{(t)}$ is the standard parabolic subgroup of $G(W)$ which fixes the flag $X_t\subseteq W$ and $S(G(W_t))$ is the regular representation of $G(W_t)$, \emph{i.e.} the set of locally constant, compactly supported functions $f\colon G(W_t)\ra \C$ on which $G(W_t)\times G(W_t)$ acts by left-right translation. We say that an irreducible smooth representation $\pi$ of $G(W)$ \emph{does not appear on the boundary} if every non-zero morphism $I_{W,W}\ra \pi\otimes\pi^\lor$ does not vanish on $I_0$.

The second ingredient is motivated by the filtration of the Jacquet module of the Weil representation in \cite[Theorem 2.8]{Kudla1986} and \cite[Proposition 3.2]{Minguez}. 
For $i,j\in\mathbb{N}$ let $\sigma_{i,j}$ be the space of locally constant, compactly supported functions $\mathrm{Hom}(F^j,F^i)\ra \C$ on which $\gl_i\times \gl_j$ acts by 
\[((g_1,g_2)\cdot f)(x)=f(g_1^{-1}xg_2).\]
Furthermore, for $r\in\NN$, $\ain{k}{0}{r}$ denote by $Q_{(k,r-k)}$ the standard parabolic subgroup of $\gl_r$ corresponding to the partition $(k,r-k)$ of $r$ and for $\ain{t}{1}{m'}$ let $P'_{(t)}$ be the parabolic subgroup fixing the flag $\langle x_{m'-t+1},\ldots, x_{m'}\rangle_E\subseteq W .$
We then prove the following.
\begin{theorem}
Let $P=P_{(r)}\times G(W)\subseteq G(W)\times G(W),\, \ain{r}{1}{m'}$ be a standard parabolic subgroup.
    The representation \[r_{P}(I_{W,W}(\chi,s))\] admits a filtration whose subquotients $\tau_{k,j},\, \ain{k}{1}{r},\, \ain{j}{r}{m'}$ admit isomorphisms
    \[A_{k,j}\colon \tau_{k,j}\iso \id_{Q_{(k,r-k)}\times P_{(j-r)}' \times P_{(j)}}^{\gl_r\times G(W_r)\times G(W)}(\underbrace{\chi\lvert\det\lvert^{s+{\frac{k}{2}}}}_{\gl_k}\otimes 
    \]\[\otimes \sigma_{r-k,j}(\chi\lvert\det\lvert^{-s-j+{\frac{r-k}{2}}}\otimes \chi\lvert\det\lvert^{s+{\frac{j}{2}}})\otimes \underbrace{\chi\lvert\det\lvert^{s+k+{\frac{j-r}{2}}}}_{\gl_{j-r}}\otimes S(G(W_j)))\]
    of $\gl_r\times G(W_r)\times G(W)$-representations,
    where $P_{(j-r)}'$ is the parabolic subgroup of $G(W_r)$ defined above.
\end{theorem}
The main idea of the proof of this theorem is the following. The representation $I_{W,W}(\chi,s)$ corresponds to a sheaf $\scf^{s,\chi}$ on the Lagrangian Grassmanian $\L_W\coloneq P(Y)\bs G(W\oplus W),$ where $P(Y)$ is the Siegel parabolic subgroup in $G(W\oplus W)$. 
We find for each $k,j$ a certain $P_{(r)}\times G(W)$-right-invariant, locally closed subset $\Gamma_{k,j}$ of $\L_W$ and show that for each point $x\in \Gamma_{k,j}$ the stabilizer of $x$ under the action of the unipotent part $N_r\times 1_W$ of $P_{(r)}\times G(W)$ is, up to conjugation, independent of $x$. This allows us to write down an explicit isomorphism from $\tau_{k,j}\coloneq r_{P_{(r)}\times G(W)}(\scf_c^{s,\chi}(\Gamma_{j,k}))$ to the representation\[\id_{Q_{(k,r-k)}\times P_{(j-r)}' \times P_{(j)}}^{\gl_r\times G(W_r)\times G(W)}(\underbrace{\chi\lvert\det\lvert^{s+{\frac{k}{2}}}}_{\gl_k}\otimes 
    \]\[\otimes \sigma_{r-k,j}(\chi\lvert\det\lvert^{-s-j+{\frac{r-k}{2}}}\otimes \chi\lvert\det\lvert^{s+{\frac{j}{2}}})\otimes \underbrace{\chi\lvert\det\lvert^{s+k+{\frac{j-r}{2}}}}_{\gl_{j-r}}\otimes S(G(W_j))).\] Here $\scf_c^{s,\chi}(\Gamma_{j,k})$ denotes the compactly supported sections on $\Gamma_{k,j}$. Moreover, since we also show that $\bigcup_{j,k}\Gamma_{k,j}=\L_W$, it follows straightforwardly from the definition of $\tau_{k,j}$ that they are the subquotients of a filtration of $r_{P_{(r)}\times G(W)}(I_{W,W}(\chi,s))$.

Finally, the third ingredient is the following theorem of  \cite{Minguez}.
\begin{theorem}[{\cite[Theorem 1]{Minguez}}]\label{T:introductionhowedualI}
Let $i,j\in\NN,\, i\le j$ and let $\pi$ be an irreducible smooth representation of $\gl_i$.
Then there exists an irreducible smooth representation $\pi'$ of $\gl_j$, unique up to isomorphism, such that
\[\ho_{\gl_i\times \gl_j}(\sigma_{i,j},\pi\otimes \pi')\neq \{0\}.\]
Moreover, for such a $\pi'$, \[\dim_\C \ho_{\gl_i\times \gl_j}(\sigma_{i,j},\pi\otimes \pi')=1.\]
\end{theorem}
Let now $\pi$ be an irreducible smooth representation of $G(W)$.
If $\pi\otimes\pi^\lor $ does not admit a morphism from any $\sigma_t$ with $t>0$, the claim follows straightforwardly. Otherwise
we can find $\ain{r}{1}{m'}$ and suitable irreducible smooth representations $\rho$ of $\gl_r$ and $\tau$ of $G(W_r)$ such that $\pi$ is a quotient of $\id_{P_{(r)}}^{G(W)}(\rho\otimes\tau)$. 
The MVW-involution then allows us to realize $\pi$ as a subrepresentation of $\id_{P_{(r)}}^{G(W)}(\rho^\lor\otimes\tau)$ and hence every morphism
\[f\colon I_{W,W}(\chi,s)\ra \pi\otimes\pi^\lor\] induces a morphism
\begin{equation*}
    f'\colon I_{W,W}(\chi,s)\ra \id_{P_{(r)}}^{G(W)}(\rho^\lor\otimes\tau) \otimes\pi^\lor.
\end{equation*}
Having done this, we can apply Frobenius reciprocity and use the filtration of \[r_{P_{(r)}\times G(W)}(I_{W,W}(\chi,s))\] to obtain some restrictions on $f'$ and subsequently on $f$.
Combining these with the first filtration \Cref{E:I1} and the theorem of \cite{Minguez}, allows us to reduce the claim of \Cref{T:mainconjecture} to the following proposition.
\begin{prop}
    Let $\sigma_t$ be as above and let $\pi$ be an irreducible smooth representation of $G(W)$. Then
    \[\dim_\C\ho_{G(W)\times G(W)}(\sigma_t,\pi\otimes\pi^\lor)\le  1.\]
\end{prop}
This can be proven by induction on $\dim_F W$ using a variant of a trick of \cite{Minguez}. Furthermore, we prove an analogous statement to \Cref{T:mainconjecture} in the case $G(W)$ being unitary or orthogonal, see \Cref{T:maintheorem} for the precise statement. 

Let us remark that to prove the Conservation relation in Type I in its full generality for above spaces $V$ and $W$, one would have to extend \Cref{T:mainconjecture} also to the case $W$ symplectic and $G(W)$ replaced by $\m(W)$, the metaplectic cover of $G(W)$.
Indeed, to handle the case where $E=F$, $W$ is symplectic and $\dim_F V=n$ is odd, the metaplectic group $\m(W)$ appears. In this case we are able to reduce the claim to an analogous statement of \Cref{T:introductionhowedualI} for metaplectic covers of the general linear group, of which we however do not have a proof at the moment, see \Cref{R:meta}. 
Finally, the case $W$ and $V$ being right $D$-vector spaces, where $D$ is a central division quaternion algebra over $F$, \emph{i.e.} the quaternionic case, has not been considered in this paper. The main obstruction here is that the MVW-involution, \emph{cf.} \cite[p.91]{MVW}, does not extend easily to this setting. 
\subsection*{Acknowledgements:}
I would like to first and foremost thank Alberto M\'inguez for suggesting to look at filtrations of the Jacquet module of the degenerate principal series and for his patience and guidance in helping me to write this article. Moreover, I would like to thank Harald Grobner, Hengfei Lu, Joachim Mahnkopf and Anton Mellit for their interesting remarks and suggestions. Finally, I am greatly indebted to the referee for reading the paper carefully
and making many useful suggestions. This work has been supported by the research projects P32333 and PAT4832423 of the Austrian Science Fund (FWF).
\section{Preliminaries}
\counterwithin{theorem}{subsection}
Let $F$ be a non-archimedean field of characteristic different from $2$, residue characteristic $p$ with absolute value $\lvert-\lvert$ and residue cardinality $q$. Let $\psi$ be a fixed non-trivial additive character of $F$ and $G$ be a reductive group or a metaplectic group over $F$. By abuse of notation we will often write $G=G(F)$. Moreover, let $E$ be either $E=F$ or a quadratic extension of $F$
and $\mathfrak{c}\in \mathrm{Gal}(E/ F)$ be the generator of $\mathrm{Gal}(E/F)$. If $[F:E]=2$, we let $E^1$ be the elements of norm $1$ in $E$. For $n\in\mathbb{N}$ we let $\zeta_n$ be the group of $n$-th roots of unity in $\C^*$.

We denote by $\mathrm{Rep}(G)$ the category of smooth representations of $G(F)$ over $\C$ and by $\mathrm{Irr}(G)$ the set of isomorphism classes of irreducible representations in $\rep(G)$.
From now on we assume all representations to be smooth. For $\pi\in \rep(G)$ we denote by $\pi^\lor$ the contragradient representation of $\pi$. If $\pi'$ is another representation of $G$, we write $\pi\hookrightarrow\pi'$ if $\pi$ is a subrepresentation of $\pi'$ and $\pi\sra\pi'$ if $\pi'$ is quotient of $\pi$. If $\iota\colon H\rightarrow G$ is a morphism and $\pi\in \rep(G)$ we write $\iota^*(\pi)$ for the pullback of $\pi$ to $H$.
For a representation of finite length $\pi\in \rep(G)$ we write $[\pi]$ for the image of $\pi$ in the Grothendieck group of representations of finite length.
Write $\Z(G)$ for the center of $G$.
And finally, let $\Delta G\subseteq G\times G$ be the diagonal of $G$.
\fakesubsection{}\label{S:2.1}
For $P=M\ltimes N$ a closed subgroup of $G$ such that $M\cap N=\{1\}$ and $M$ normalizes $N$, and $(\tau,V)$ a representation of $M$ we let $\#-\id_P^G(\tau)$ be the compactly supported induction of $\tau$ to $G$. The underlying vector space is the set of all functions
$f\colon G\ra V$ satisfying 
\begin{enumerate}
    \item $f(mng)=\tau(m)f(g)$ for all $m\in M,\, n\in N$ and $g\in G$,
    \item there exists an open compact subgroup $K_f\subseteq G$ such that for all $k\in K_f$ and $g\in G$
    $f(gk)=f(g)$,
    \item $f$ is compactly supported modulo $P$.
\end{enumerate}
The group $G$ acts then on $\#-\id_P^G(\tau)$ by right translations. We denote by $\id_P^G(\tau)=\id_P(\tau)$ the normalized compactly-supported induction of $\tau$, \textit{i.e.}
\[\id_P^G(\tau)\coloneq \#-\id_H^G(\delta_P^{\frac{1}{2}}\tau),\] where $\delta_P$ is the modular character of $P$.
 If $P=M\ltimes N$ is a parabolic subgroup with respective Levi-decomposition, $P\bs G$ is compact and therefore the third condition on the functions $f\in \id_P^G(\tau)$ is superfluous. In this case we call this induction {parabolic induction}.
For $(\pi,V)$ a representation of $G$, we denote by $\#-r_P(\pi)$ the reduction of $\pi$ to $M$. To be more precise, let $V^N\subseteq V$ be the subspace of $V$ spanned by the vectors of the form
\[\{\pi(n)v-v,\text{ for }n\in N, v\in V\}\]
and let $V_N\coloneq V/V^N$. Then $\pi$ restricted to $M$ gives a well-defined action of $M$ on $V_N$, which is by definition the reduction of $V$.
We denote by $r_P(\pi)$ normalized reduction, \textit{i.e.} \[r_P(\pi)\coloneq \delta_P^{-{\frac{1}{2}}}(\#-r_P(\pi)).\]
If $P=M\ltimes N$ is a parabolic subgroup of $G$ we call normalized reduction the \emph{Jacquet-functor}, the image of a representation under the Jacquet-functor its \emph{Jacquet module}, and
we obtain functors
\[\id_P^G\colon \rep(M)\rightarrow \rep(G),\, r_P\colon \rep(G)\rightarrow \rep(M),\] which
are exact, send representations of finite length to representations of finite length and satisfy $\id_P^G(\tau)^\lor\cong \id_P^G(\tau^\lor)$ and $r_{\overline{P}}(\tau^\lor)\cong r_P(\tau)^\lor$, where $\overline{P}$ denotes the opposite parabolic subgroup of $P$. Moreover,
\[\ho_M(r_P(\pi),\tau)=\ho_G(\pi,\id_P^G\tau)\text{, (Frobenius reciprocity)}\]
\begin{equation}\label{E:Bernstein}\ho_G(\id_P^G(\tau),\pi)=\ho_M(\tau,r_{\overline{P}}(\pi))\text{, (Bernstein reciprocity).}\end{equation}
for all $\pi\in \rep(G)$ and $\tau\in \rep(M)$.
We call an irreducible representation $\pi\in\ir(G)$ \emph{cuspidal} if $r_{P}(\pi)=0$ for all nontrivial parabolic subgroups $P$ of $G$.
\begin{lemma}\label{L:irreduciblesubquotient}
    Let $P$ be a parabolic subgroup of $G$ with Levi-decomposition $P=M\ltimes N$, $\pi\in \rep(M)$ not necessarily of finite length and
    $\tau\in\rep(G)$ a representation of finite length. Let $f\colon \id_P^G(\pi)\ra\tau$ be a non-zero morphism.
    Then there exists an irreducible subquotient $\sigma$ of $\pi$ and a non-zero morphism $ \id_P^G(\sigma)\ra \tau$.
\end{lemma}
\begin{proof}
Assume first $G=P$ and take $\sigma$ to be an irreducible subrepresentation of $\mathrm{Ker}(f)/ \pi\hra \tau$. Since $f$ factors through $\pi\sra \mathrm{Ker}(f)/ \pi$ we obtain the desired morphism by restricting to $\sigma$.
For general $P$, we obtain by Frobenius reciprocity a non-zero morphism $\pi\ra r_P(\tau)$. Applying the first case to $M$ and $P=M$ yields an irreducible subquotient $\sigma$ of $\pi$ together with a non-zero morphism $\sigma\ra r_P(\tau)$. Again by Frobenius reciprocity, we obtain a non-zero morphism $ \id_P^G(\sigma)\ra \tau$.
\end{proof}
Let $P$ be a parabolic subgroup of $G$ with Levi-component $M=M_1\times M_2$ and $\rho$ a representation of $M_1$. If $\pi$ is a representation of $G$, denote by $\mathrm{Jac}_\rho(\pi)$ the $(M_1,\rho)$-invariant vectors of  $r_{\overline{P}}(\pi)$, \emph{i.e.} the maximal representation $\sigma$ such that $\rho\otimes\sigma\hra r_{\overline{P}}(\pi)$. 
\fakesubsection{}\label{S:sheaf}
We will now quickly recap the theory of $\ell$-spaces and their sheaves. For a precise treatment see \cite{Bernstein2}.

A topological space $X$ is called an $\ell$-space if it is Hausdorff and the open compact sets form a base of the topology. The ring of smooth, \textit{i.e.} locally constant functions on $X$ is denoted by $C^\infty(X)$ and the ring of compactly supported smooth functions, also called Schwartz-functions, is denoted by $S(X)$. 
An $\ell$-sheaf on $X$ is a sheaf $\scf$ which is a module over the sheaf of smooth functions. The category of $\ell$-sheaves on $X$ is denoted by $\sh(X)$. For $\scf\in \sh(X)$ we let $\scf(X)$ be its global sections and $\scf_c(X)$ its compactly supported global sections.
\begin{prop}[\cite{Bernstein2} Proposition 1.14]\label{P:sheafequi}
The functor $\scf\mapsto \scf_c(X)$ induces an equivalence of categories from $\sh(X)$ to $S(X)$-modules $M$ such that \[S(X)\cdot M=M.\]
\end{prop}
\begin{prop}[\cite{Bernstein2} 1.16]\label{P:shortexact}
If $U$ is an open subspace of $X$, $Z\coloneq X\bs U$ and $\scf\in \sh(X)$, there is a short exact sequence 
\[0\rightarrow \scf_c(U)\rightarrow \scf_c(X)\rightarrow \scf_c(Z)\rightarrow 0\]
\end{prop}
An $\ell$-group $G$ is a topological group $G$, which is an $\ell$-space. For example, if $G$ is a reductive group over a $F$, then its $F$-points are an $\ell$-group.
An action of an $\ell$-group $G$ on an $\ell$-sheaf $\scf\in \sh(X)$ is a continous action of $G$ on $X$ and a morphism
\[\gamma\colon G\rightarrow \mathrm{Aut}(X,\scf),\]
such that $G$ acts on $\scf_c(X)$ smoothly. For a fixed action $\gamma_0$ of $G$ on $X$, we let $\sh(X,G)$ be the category of $G$-sheaves, where the action of $G$ restricted to $X$ is $\gamma_0$.
We denote the functor $\scf\mapsto \scf_c(X)$ by
\[\scc\colon \sh(X,G)\rightarrow \rep(G).\]
Moreover, if $Q$ is closed subgroup of $G$ and $Z$ a locally closed and $Q$-invariant subspace of $X$ there exists a restriction functor
\[\rs=\rs_{Z,Q}\colon \sh(X,G)\rightarrow \sh(Z,Q).\]
\begin{prop}[\cite{Bernstein2} Proposition 2.23]\label{P:sheafin}
Let $P$ be a closed subgroup of $G$ and set $X=P\bs G$. Then the functor 
\[\rs\colon \sh(X,G)\rightarrow \sh({*},P)=\rep(P)\] has an inverse, denoted by
\[\id\colon \sh({*},P)\rightarrow \sh(X,G).\]
Moreover, $\scc\circ\id\colon \rep(P)\rightarrow \rep(G)$ is $\#-\id_P^G$.
\end{prop}
Let $P=M_P\ltimes N_P,\, Q=M_Q\ltimes N_Q$ be closed subgroups of $G$  with $M_P\cap N_P=\{1\}$, $M_P$ normalizes $N_P$ and similarly for $Q$ such that there only finitely many $Q$-orbits in $P\bs G$.
 Define for $\mathfrak{w}\in P\bs G/Q$ the groups
 \[M_P'\coloneq M_P\cap w^{-1}M_Qw,\, M_Q'\coloneq wM_P'w^{-1},\,N_Q'\coloneq M_P\cap w^{-1}N_Qw,\, N_P'\coloneq M_Q\cap w N_Pw^{-1},\]
where $w$ is a representative of $\mathfrak{w}$. Let \[\delta_1\coloneq \delta_{N_P}^{\frac{1}{2}}\cdot\delta_{N_P\cap w^{-1}Qw}^{-{\frac{1}{2}}},\, \delta_2\coloneq\delta_{N_Q}^{\frac{1}{2}}\cdot\delta_{N_Q\cap wPw^{-1}}^{-{\frac{1}{2}}}\] be characters of $M_P'$ respectively $M_Q'$. Finally, let $w\colon \rep(M_P')\ra \rep(M_Q')$ be the pullback by conjugation by $w$ and set $\delta\coloneq \delta_1w^{-1}(\delta_2)$.
Order the $Q$-orbits of $P\bs G$ as  $\mathfrak{w}_1,\ldots, \mathfrak{w}_l$ such that $\overline{O_{\mathfrak{w}_i}}=\bigcup_{i\le j}O_{\mathfrak{w}_j},$
where $O_\mathfrak{w}$ is the $Q$-orbit in $P\bs G$ associated to $\mathfrak{w}$.

 Define for $\mathfrak{w}\in P\bs G/Q$ the functor \[F({w})\coloneq \id_{M_P'N_P'}^{M_Q}\circ w\circ \delta\circ r_{M_Q'N_Q'}\colon \rep(M)\ra\rep(N).\]
Then the following holds.
\begin{lemma}[\cite{Ber} Lemma 2.11]
The functor \[F\coloneq r_Q\circ \id_P^G\colon \rep(M)\ra\rep(N)\] has a filtration $0=F_0\subseteq F_1\subseteq F_2\subseteq\ldots\subseteq F_l=F$ with subquotients $F_{i-1}\bs F_i\cong F({w}_i)$. 
\end{lemma}
Let us comment on its proof, as the details will be important later on.
An induced representation $\id_P^G(\tau)$ corresponds to $G$-sheaf $\scf$ on $P\bs G$. Note that we have a filtration of $P\bs G$ by \[\emptyset\subseteq O_{\mathfrak{w}_1}\subseteq \ldots\subseteq\bigcup_{j\le i}O_{\mathfrak{w}_j}\subseteq\ldots\subseteq \bigcup_{j\le l}O_{\mathfrak{w}_j}=P\bs G\] 
inducing a filtration of $Q$-representations
\[0\subseteq \scf_c(O_{\mathfrak{w}_1})\subseteq \ldots\subseteq\scf_c(\bigcup_{j\le i}O_{\mathfrak{w}_j})\subseteq\ldots\subseteq  \id_P^G(\tau).\]
We then set $\tau_\mathfrak{w}\coloneq r_Q(\scf_c(O_\mathfrak{w}))$, which are the subquotients of the above filtration  of $r_Q(\id_P^G(\tau))$ by \Cref{P:shortexact}.
One can then define the isomorphism
\[A_\mathfrak{w}\colon \tau_\mathfrak{w}\ra F({w})(\tau) \] by
\[A_\mathfrak{w}([f])(m)= \int_{N_P'\bs N_Q}p(f(wmn))\,\mathrm{d}n, \]
where $m\in M_Q$, $\mathrm{d}n$ is a Haar-measure on $N_Q$, $[f]$ is the equivalence class of a section in $\tau_\mathfrak{w}$ and $p$ is the projection $p\colon \tau \sra r_{N_Q'}(\tau)$.
\fakesubsection{}
We will now recall some facts about representations of $\gl_n(F)$.
Let $V$ be an $n$-dimensional vector space over $E$. We denote by $\gl(V)$ the group of linear automorphisms of $V$. Choosing a basis $e_1,\ldots,e_n$ of $V$ gives an identification
$\gl(V)\cong\gl_n\coloneq \gl(E^n)$.
For elements $f_1,\ldots,f_m$ in $V$ we denote by $\langle f_1,\ldots,f_m\rangle_E$ their linear span and we set $V_r\coloneq \langle e_1,\ldots,e_r\rangle_E$. We denote by $1_n= 1_V$ the identity element in $\gl(V)$.
\fakesubsubsection{}\label{S:metaplectic}
Assume for the moment $E=F$.
We set $\wt{\gl(V)}$ to the twofold cover of $\gl(V)$ which is 
\[\gl(V)\times \zeta_2\]
with multiplication \[(g,\zeta)\cdot (g',\zeta')=(gg',\zeta\zeta'\cdot(\det(g),\det(g'))_F),\]
where $(,)_F$ is the Hilbert symbol of $F$.
There exists a bijection \[(-)^\psi\colon {\rep(\gl(V))\ra\rep(\wt{\gl(V)})},\] which depends on our chosen additive character $\psi$ of $F$ and is constructed as follows.
First lift $\det\colon \gl(V)\rightarrow \gl_1$ to \[\wt{\det}\colon \wt{\gl(V)}\ra\wt{\gl(F)}\] by $(g,\zeta)\mapsto (\det(g),\zeta)$.
Let $\gamma_F(\psi)\in \zeta_8$ be the Weil index of $\psi$.
For $a\in F$, let $\psi_a\coloneq \psi(a\,\cdot)$ and set
\[\gamma_F(a,\psi)\coloneq \frac{\gamma_F(\psi_a)}{ \gamma_F(\psi)}.\]
This gives a character of $\widetilde{\gl(F)}\ra \C^*$ sending
$(a,\zeta)\mapsto \zeta\gamma_F(a,\psi)^{-1}$.
Composing with $\wt{\det}$ gives a character $\chi_\psi$ of $\wt{\gl(V)}$.
Finally, we define the bijection $\rep(\gl(V))\ra\rep(\wt{\gl(V)})$ by sending \[\pi\mapsto \pi^\psi\coloneq \pi\otimes\chi_\psi.\]
It is clear that it commutes with parabolic induction and reduction. For $g\in \wt{\gl(V)}$ we let $^t(g,\zeta)\coloneq ({}^tg,\zeta)$.

From now on we mean by $H(V)$ either $\gl(V)$ or $\wt{\gl(V)}$ and write ${H_n\coloneq H(F^n)}.$
If $\chi$ is a character of $F^*$ we will write by abuse of notation $\chi$ also for the representation $\chi\circ\mathrm{Nrm}_{E/F}\circ\det$ respectively $\chi\circ\wt{\det}$ or $\chi\circ\det$ of $H(V)$.
\fakesubsubsection{}\label{E:parabolicgl}
For a partition $\alpha=(\alpha_1,\ldots,\alpha_t)$ of $n$ we let $Q_\alpha$ be the stabilizer in $H_n$ of the flag
\[0\subseteq V_{\alpha_1}\subseteq V_{\alpha_1+\alpha_2}\subseteq \ldots\subseteq V_n= V.\]
Then the Levi subgroup $M_\alpha$ of $Q_\alpha$ is of the form 
$M_\alpha\cong  H_{\alpha_1}\times\ldots\times  H_{\alpha_t}$. In the metaplectic case we take the product over $\zeta_2$.
For $\tau=\tau_1\otimes\ldots\otimes \tau_t\in \rep(M_\alpha)$ we denote
\[\tau_1\times\ldots\times\tau_t\coloneq\id_{Q_\alpha}^{H_n}\tau\]
and $r_{\alpha}\coloneq r_{Q_\alpha}$.
We take the opportunity to note at this point that \[\delta_{Q_\alpha}={\underbrace{\lvert-\lvert}_{H_{\alpha_1}}} ^{q(\alpha)_1}\otimes\ldots\otimes {\underbrace{\lvert-\lvert}_{H_{\alpha_t}}} ^{q(\alpha)_t},\]
with $q(\alpha)_i\coloneq -(\sum_{j=1}^{i-1}\alpha_j)+(\sum_{j=i+1}^t\alpha_j)$ and that $\overline{Q_\alpha}$ is conjugated to $Q_{\overline{\alpha}}$, where $\overline{\alpha}\coloneq(\alpha_t,\ldots,\alpha_1).$ For later use, we define for $a,b\in \NN$, $a+b\le n$ the parabolic subgroup $Q'_{(a,b,n-a-b)}\subseteq H_n$ as the stabilizer of the flag $\langle e_1,\ldots,e_a\rangle_E \subseteq \langle e_1,\ldots e_a,e_{n-b-a+1},\ldots e_n\rangle_E\subseteq E^n.$
For $\pi\in\rep( H_n),\, k\in\NN$ we also denote by $\pi^k\coloneq \overbrace{\pi\times\ldots\times \pi}^{k}\in \rep( H_{kn})$.
If $\pi$ is an irreducible representation of $\gl(V)$, we define a representation ${}^\mathfrak{c}\pi$ as follows. Choose an isomorphism  $\alpha\colon V\tilde{\ra} E^n$ and define $\mathfrak{c}(g)\coloneq \alpha^{-1}(\mathfrak{c}(\alpha(g)))$, where $\mathfrak{c}(\alpha(g))$ is the natural action of $\mathfrak{c}$ on $\gl(E^n)$.
We then set ${}^\mathfrak{c}\pi(g)\coloneq \pi(\mathfrak{c}(g))$. The isomorphism class of ${}^\mathfrak{c}\pi$ is independent of the chosen isomorphism $\alpha$, since all of those differ by an inner automorphism of $V$.
Observe that if $\pi$ is a character then ${}^\mathfrak{c}\pi\cong \pi$.
Finally, if $\pi$ is an irreducible representation of $H_n$, the pullback of $\pi$ along the morphism $g\mapsto {}^tg^{-1}$ is isomorphic to $\pi^\lor$, see \cite[Theorem 7.3]{Ber}.
\fakesubsubsection{}
We recall the following well-known facts about induced and irreducible representations of $ H_n$, see \cite[Theorem 1.9, Theorem 4.1, Proposition 4.6 and Theorem 6.1]{Zel} for the case $H_n=\gl_n$ and \cite[ § 7]{KLZ} for the case $H_n=\widetilde{\gl_n}$.
\begin{lemma}\label{L:Ggroupcommutative}
Parabolic induction is commutative on the Grothendieck ring of $ H_n$, \textit{i.e.} for $\ain{r}{0}{n}$, $\pi\in \rep( H_r), \,\pi'\in\rep( H_{n-r})$, \[[\pi\times\pi']=[\pi'\times \pi].\]
\end{lemma}
Note that that if $\pi\times \pi'$ is irreducible, this implies that $\pi\times \pi'\cong \pi'\times \pi$.
\begin{lemma}\label{L:glirr}
Let $\rho, \rho'$ be cuspidal representations of $ H_m$ and $H_{m'}$. Then $\rho\times \rho'$ is a reducible representation of $ H_{m+m'}$ if and only if $\rho'\cong \rho \lvert-\lvert^{\pm1}$. 
For $\rho_1\otimes\ldots\otimes \rho_k\in \ir(H_{m_1}\times\ldots\times H_{m_k})$ cuspidal the following two statements are equivalent. \begin{enumerate}
    \item For all $\ain{i,j}{1}{k}$, $i\neq j$ $\rho_i\ncong \rho_j\lvert-\lvert^{\pm 1}$.
    \item The representation $\rho_1\times\ldots\times\rho_k$
is irreducible.
\end{enumerate}
    If $\rho_i=\rho\lvert-\lvert^i$, then $\rho\times\rho_1\times\ldots\times \rho_k$  has as a unique subrepresentation which we denote as $\Z([0,k]_\rho)$.
\end{lemma}
More generally, for $a\le b$ integers and $\rho$ a cuspidal representation, we define the segment $[a,b]_\rho$ as the sequence 
\[[a,b]_\rho\coloneq (\rho\lvert-\lvert^a,\ldots,\rho\lvert-\lvert^b).\]
The length of $[a,b]_\rho$ is defined as $l([a,b]_\rho)\coloneq b-a+1$.
To each segment $\Delta=[a,b]_\rho$ we associate a representation 
\[\Z(\Delta)\coloneq \Z([0,b-a]_{\rho\lvert-\lvert^a}).\]
\begin{lemma}\label{L:glred}
If $a'+b'=(k+1)m$ with $a'=am,\, b'=bm$, then \[r_{Q_{(a',b')}}(\Z([0,k]_\rho))\cong \Z([0,a-1]_\rho)\otimes \Z([a,k]_\rho),\, r_{\overline{Q_{(a',b')}}}(\Z([0,k]_\rho))\cong  \Z([a,k]_\rho)\otimes \Z([0,a-1]_\rho).\]
If $a'+b'=(k+1)m$ and $a'$ and $b'$ are not divisible by $m$, then 
\[r_{{Q_{(a',b')}}}(\Z([0,k]_\rho))=r_{\overline{Q_{(a',b')}}}(\Z([0,k]_\rho))=0.\]
Finally, if $\rho=\lvert-\lvert^{-{\frac{k}{2}}}$, 
$\Z([0,k]_\rho)=1$ is the trivial representation of $H_{k+1}$.
\end{lemma}
We say $\Delta=[a,b]_\rho$ precedes $\Delta'=[a',b']_{\rho'}$ if 
the sequence
\[(\rho\lvert-\lvert^a,\ldots,\rho\lvert-\lvert^b,\rho'\lvert-\lvert^{a'},\ldots,\rho'\lvert-\lvert^{b'})\] contains a subsequence which, up to isomorphism, is a segment of length greater than $l(\Delta)$ or $l(\Delta')$. We call $\Delta$ and $\Delta'$ unlinked if $\Delta$ does not precede $\Delta'$ and vice versa.
\begin{lemma}\label{L:glirr2}
The representations $\Z([a,b]_\rho)$ and $ \Z([a',b']_{\rho'})$ are isomorphic if and only if
$b-a=b'-a'$ and $\rho\lvert-\lvert^a\cong \rho'\lvert-\lvert^{a'}$.
Moreover, if $\Delta_1,\ldots,\Delta_k$ are pairwise unlinked segments then the representation
\[\Z(\Delta_1)\times\ldots\times \Z(\Delta_k)\]
is irreducible.
Finally,
\[\Z([a,b]_\rho)^\lor\cong\Z([-b,-a]_{\rho^\lor}).\]
\end{lemma}
Observe that this implies, together with the commutativity of $\times$ on the Grothendieck group, that if $\Delta_1,\ldots,\Delta_k$ are pairwise unlinked segments,
the isomorphism class of \[\Z(\Delta_1)\times\ldots\times \Z(\Delta_k)\] does not depend on the order of the segments.
\begin{lemma}\label{L:quotientglsegments}
    Let $\Delta_\rho=[a,b]_\rho$ be a segment and $a\le c\le b$ be an integer.
    Then $\Z(\Delta)$ is the unique subrepresentation of \[\Z([a,c]_\rho)\times \Z([c+1,b]_\rho)\] and the unique quotient of \[\Z([c+1,b]_\rho)\times \Z([a,c]_\rho).\]
\end{lemma}
\fakesubsubsection{}
Let $\ain{a}{1}{n},\,\pi\in \rep( H_a),\, \pi'\in \rep( H_{n-a})$. We will give a combinatorial description of the Geometric Lemma for $r_{(r,n-r)}(\pi\times \pi')$, \emph{cf.} \cite[§ 1.6]{Ber}. Let $k\in \mathbb{N}$ such that $k\le a$ and $r-k\le n-a$. Write the semisimplification of $r_{(k,a-k)}(\pi)$ as the sum of irreducible representations of the form $\pi_1\otimes \pi_2$ and the semisimplification of $r_{(r-k,n-a-r+k)}(\pi')$ as the sum of irreducible representations of the form $\pi_3\otimes \pi_4$. Then 
\[[r_{(r,n-r)}(\pi\times \pi')]=\sum [\pi_1\times \pi_3\otimes \pi_2\times \pi_4],\]
where the sum is over all $k$ and $\pi_1,\pi_2,\pi_3,\pi_4$ as above.
\fakesubsection{}
Let $\epsilon\in\{\pm 1\}$ and $W$ a finite dimensional vector space over $E$ together with a non-degenerate $\epsilon$-hermitian and $\mathfrak{c}$-sesquilinear form  \[\langle.,.\rangle\colon W\times W\ra E,\]
\[\langle \lambda x+\mu y,z\rangle=\lambda\langle x,z\rangle+\mu\langle y,z\rangle,\,\langle x,y\rangle=\epsilon\mathfrak{c} (\langle y,x\rangle).\]
Let $G(W)$ be the symmetry group of $\langle.,.\rangle$, \textit{i.e.} \[G(W)\coloneq\{g\in \gl(W):\langle gx,gy\rangle=\langle x,y\rangle\text{ for all }x,y\in W\}.\] Then
\[G(W)=\begin{cases}
\text{symplectic group }\mathrm{Sp}(W)& E=F,\epsilon=-1.\\
\text{orthogonal group }\oo(W)& E=F, \epsilon=1.\\
\text{unitary group }\uu(W)& [F:E]=2.\end{cases}\]
We call a subspace $X$ of $W$ \emph{isotropic} if the $\epsilon$-hermitian form vanishes on $X$.
Let $n\coloneq \dim_E W$ and $q_W$ be Witt index of $W$, \textit{i.e.} the maximal dimension of an isotropic subspace of $W$.
If $X$ is an isotropic subspace of $W$, write $W=X\oplus W'\oplus Y$, such that $X\oplus Y$ and $W'$ are non-degenerate and $Y$ is isotropic. Let $P(X)$ be the stabilizer of $X$ in $G(W)$. Then $P(X)$ is a maximal parabolic subgroup of $G(W)$ and every maximal parabolic subgroup is of this form. The Levi decomposition of $P(X)=M(X)\ltimes N(X)$ has Levi-component $M(X)=\gl(X)\times G(W')$, the stabilizer of $Y$ in $P(X)$. More generally, if \[\mathcal{F}=\{0\subseteq X_1\subseteq \ldots\subseteq X_r\}\] is a flag of isotropic subspaces of $W$, its stabilizer is a parabolic subgroup $P(\mathcal{F})=M(\mathcal{F})\ltimes N(\mathcal{F})$ with Levi-component 
\[M(\mathcal{F})=\gl(X_1)\times\ldots\times\gl(X_r)\times G(W'),\] where $W=X_r\oplus W'\oplus Y_r$ as above.
We observe here that the parabolic subgroup $P(X)$ is conjugated to its opposite parabolic subgroup $\overline{P(X)}$ by an element which acts on $\gl(X_i)$ as $g\mapsto \mathfrak{c}({}^tg^{-1})$ and on $G(W')$ trivially.
\fakesubsubsection{}
In the case $E=F,\, \epsilon =-1$ we will also treat the metaplectic group $\m(W)$, which sits in the short exact sequence
\[0\ra \zeta_2\ra \m(W)\ra\s(W)\ra0.\]
For \[\mathcal{F}=\{0\subseteq X_1\subseteq \ldots\subseteq X_r\}\] a flag of isotropic subspaces of $W$ and ${W=X_r\oplus W'\oplus Y_r}$ as above, let \[\wt{P(\mathcal{F})}\coloneq \wt{M(\mathcal{F})}\ltimes N(\mathcal{F}),\]
where $\wt{M(\mathcal{F})}$ is the inverse image of $M(\mathcal{F})$ in $\m(W)$. Since the preimage of $N(\mathcal{F})$ in $\m(W)$ is split, we can see it as a subgroup of $\m(W)$.
Then the Levi-component of $\wt{P(\mathcal{F})}$ is \[\wt{M(\mathcal{F})}=\wt{\gl(X_1)}\times_{\zeta_2}\ldots\times_{\zeta_2} \wt{\gl(X_r)}\times_{\zeta_2} \m(W').\]
From now on $G(W)$ can either mean $\s(W)$ or $\m(W)$ in the case ${E=F, \epsilon=-1}$.
If $\chi$ is a character of $F^*$, we lift $\chi$ to $G(W)$ by composing with $\det$ respectively $\wt{\det}$ and restricting to the center of $G(W)$. 
More explicitly, \[\det(\Z(G(W)))=\begin{cases}1&W\text{ symplectic or metaplectic,}\\
\{\pm1\}&W\text{ orthogonal and }n\text{ odd,}\\
1&W\text{ orthogonal and }n\text{ even,}\\
E^1&W\text{ unitary.}
\end{cases}\]
\fakesubsubsection{}\label{S:parabolicsp}
Fix now $X$ a maximal isotropic subspace of $W$, $W=X\oplus W'\oplus Y$ as above and choose a basis of $\{ e_1,\ldots,e_{q_W}\}$ of $X$ and a basis $\{ f_1,\ldots,f_{q_W}\}$ of $Y$ 
such that \[\langle e_i,f_j\rangle=\begin{cases}
    1 &\text{ if } i=j,\\ 
    0 &\text{ if } i\neq j.\\
\end{cases}\]
Extend then $e_1,\ldots,e_{q_W},f_1,\ldots,f_{q_W}$ to a basis of $W$ and consider $G(W)$ in this basis.
For $\ain{k}{1}{q_W}$ let \[X_k\coloneq \langle e_1,\ldots,e_k\rangle_E,\, Y_k\coloneq\langle f_1,\ldots,f_k\rangle_E \] and
write $W_k\coloneq W'$ for the non-degenerate
part in the decomposition \[W=X_k\oplus W'\oplus Y_k.\]
More generally, for $a\le b,\, \ain{a,b}{0}{q_W}$ we write $X_{a,b}$ for the subspace of $X$ spanned by $\{e_{a+1},\ldots,e_b\}$ and similarly for $Y_{a,b}$.
For a partition $\alpha$ of $\ain{k}{1}{q_W}$
let $P_\alpha$ be the stabilizer of the flag
\[0\subseteq X_{\alpha_1}\subseteq X_{\alpha_1+\alpha_2}\subseteq \ldots\subseteq X_k\]
with Levi-component $M_\alpha$.
For $\tau=\tau_1\otimes\ldots\otimes \tau_t\otimes \sigma\in \rep(M_\alpha)$ we denote
\[\tau_1\times\ldots\times\tau_t\rtimes\sigma \coloneq\id_{P_\alpha}^{G(W)}\tau.\]
We take the opportunity to note at this point that \[\delta_{P_\alpha}={\underbrace{\lvert-\lvert}_{H_{\alpha_1}}} ^{p(\alpha)_1}\otimes\ldots\otimes {\underbrace{\lvert-\lvert}_{H_{\alpha_t}}} ^{p(\alpha)_t},\]
where $p(\alpha)_i\coloneq n- (2\sum_{j=1}^{i-1}\alpha_j)-\alpha_i-\eta,$
where \[\eta\coloneq\begin{cases}
    \epsilon& \text{ if } E=F,\\
    0& \text{ if } [F:E]=2.
\end{cases}\]
Let $\ain{k}{1}{q_W}$.
An element $(m,g)\in H_k\times G(W_k)$ of the Levi-component of of $P(X_k)$ corresponds to an element of the form
\[
\begin{pmatrix}
    m&0&0\\
    0&g&0\\
    0&0&\mathfrak{c}({}^tm^{-1})\\
\end{pmatrix}\in G(W) 
\]
and an element $n\in N_k$ in the unipotent part of $P(X_k)$ is of the form
\[\begin{pmatrix}
    1&x&y-{\frac{1}{2}}x\mathfrak{c}({}^tx)\\
    0&1&-\mathfrak{c}({}^tx)\\
    0&0&1\\
\end{pmatrix}\in G(W)\]
for  $x\in \ho(W_k,X_k)$ and $y\in \ho(Y_k, X_k)$ with $\mathfrak{c}({}^ty)=y$.
Finally, we denote by $P'_{(k)}$ the stabilizer of the flag $X_{q_W-k,q_W}$.
\fakesubsection{}\label{S:MVW}
 Recall the M{\oe}glin-Vign\'eras-Waldspurger involution $\mathrm{MVW}\colon \rep(G(W))\rightarrow \rep( G(W))$, see \cite[p.91]{MVW}, which is covariant, exact and satisfies
\begin{enumerate}
    \item $\mathrm{MVW}\circ\mathrm{MVW}=\mathrm{id}$,
    \item $\pi\mvw\cong \pi^\lor$ if $\pi$ is irreducible,
    \item For $\ain{r}{1}{q_W}$, $\alpha=(\alpha_1,\ldots,\alpha_k)$ a partition of $r$, $\tau_i\in \ir(H_{\alpha_i}$, $\sigma\in\ir(G(W_r))$,\[(\tau_1\times\ldots\times \tau_k\rtimes \sigma)\mvw\cong{}^\mathfrak{c}\tau_1\times\ldots\times {}^\mathfrak{c}\tau_t\rtimes\sigma\mvw.\]
\end{enumerate}
As a consequence, if $\ain{t}{1}{q_W}$, $\pi\in \ir( G(W)),\, \rho\in \rep( H_t),\, \delta\in \ir( G(W_t))$
 the following are equivalent.
\[ \pi\hra \rho\rtimes\delta\Leftrightarrow 
      \pi^\lor\hra {}^\mathfrak{c}\rho\rtimes \delta^\lor\Leftrightarrow
      {}^\mathfrak{c}\rho^\lor\rtimes\delta\sra \pi\Leftrightarrow
      \rho^\lor\rtimes\delta^\lor\sra \pi^\lor.\]
      Indeed, the first equivalence follows from the covariance of $\mathrm{MVW}$ and properties 2 and 3. The second equivalence follows from duality and the third follows again from the covariance of $\mathrm{MVW}$ and properties 2 and 3.
\fakesubsection{}\label{S:geolemsp}
Let $\ain{t,r}{1}{q_W}$,
$\pi\in \rep( H_t)$ and $\sigma\in \rep( G(W_t))$. We will give a more combinatorial description of the Geometric Lemma for $r_{P_{(r)}}(\pi\rtimes \sigma)$ of \cite[Lemma 5.1]{TADIC19951}. Let $k_1, k_2, k_3\in \mathbb{N}$ such that $k_1+k_2+k_3=t$ and $k_1+k_3\le r,\, r+k_2\le q_W$. Then the $P_{(r)}$-orbits in $P_{(t)}\bs G(W)$ are indexed by above triples and a $t$-dimensional isotropic subspace $U\in P_{(t)}\bs G(W)$ is contained in the orbit corresponding to $k_1,k_2,k_3$ if and only if $\dim_E(U\cap X_r)=k_1$ and $\dim_E(p_{Y_r}(U))=k_3,$ where $p_{Y_r}$ is the projection to $Y_r$.
We then get a description of the semisimplified version of the Geometric Lemma as follows.
Write the semisimplification of $r_{(k_1,k_2,k_3)}(\pi)$ as the sum of irreducible representations of the form $\pi_1\otimes \pi_2\otimes \pi_3$ and the semisimplification of $r_{P_{(r-k_1-k_2)}}(\sigma)$ as the sum of irreducible representations of the form $\pi_4\otimes \sigma'$. Then 
\[[r_{P_{(r)}}(\pi\rtimes \sigma)]=\sum [ \pi_1\times \pi_4 \times {}^\mathfrak{c}\pi_3^\lor\otimes \pi_2\rtimes \sigma'],\]
where the sum is over all $k_1,k_2,k_3$ and $\pi_1,\pi_2,\pi_3,\pi_4,\sigma'$ as above. 
\fakesubsubsection{}\label{S:alphabeta}
We will now give an explicit representative $w_{k_1,k_2,k_3}$ of the $P_{(r)}$-orbit in $P_{(t)}\bs G(W)$ corresponding to the triple $k_1,k_2,k_3$.
For $\ain{k}{1}{q_W}$ we set 
\[\alpha_k \coloneq \begin{pmatrix}
    0&1_{k}&0&0\\    
    1_{k}&0&0&0\\
    0&0&0&1_{k}\\
0&0&1_{k}&0\\
\end{pmatrix}\in  G(X_k\oplus X_k\oplus Y_k\oplus Y_k)\]
and
\[\beta_k\coloneq \begin{cases}\begin{pmatrix}
    0&1_{k}\\
    \epsilon 1_{k}&0\\
\end{pmatrix}&\text{ if } G(X_k\oplus Y_k)\text{ is symplectic or orthogonal}, \\
\left(\begin{pmatrix}
    0&1_{k}\\
    -1_{k}&0\\
\end{pmatrix},1\right)&\text{ if } G(X_k\oplus Y_k)\text{ is metaplectic},\\
\begin{pmatrix}
    0&0& 1_{k}&0\\
    0&0&0&-1_{k}\\
    \epsilon 1_k&0&0&0\\
     0&-\epsilon 1_{k}&0&0\\
\end{pmatrix}&\text{ if } G(X_k\oplus Y_k)\text{ is unitary},\end{cases}
\]
where in the last case we fix a basis $\langle x_1',\ldots, x_k', \gamma x_1',\ldots,\gamma x_k'\rangle_E$ of $X_k$ with $\gamma\notin F$ and similarly for $Y_k$.
Set $a\coloneq \min(t,r)$.
For a triple $k_1,k_2,k_3$ decompose $W$ as
\[W=W_1\oplus W_2\oplus W_3\oplus W_4\oplus W_5\oplus W_6\oplus W_7,\]\[ W_1\coloneq X_{k_1}\oplus Y_{k_1},\, W_2\coloneq X_{k_1,k_1+k_3}\oplus Y_{k_1,k_1+k_3}\]
\[W_3\coloneq X_{k_1+k_3,a}\oplus X_{q_W-a+k_1+k_3,q_W},\, W_4\coloneq    Y_{k_1+k_3,a}\oplus Y_{q_W-a+k_1+k_3,q_W},\]
\[W_5\coloneq X_{a,t}\oplus X_{q_W-k_2,q_W-a+k_1+k_3},\, W_6\coloneq    Y_{a,t}\oplus Y_{q_W-k_2,q_W-a+k_1+k_3},\]
\[  W_7\coloneq
        X_{t,q_W-k_2}\oplus Y_{t,q_W-k_2}\oplus W_{q_W}.
\]
We have a natural embedding
\[\iota_{k_1,k_2,k_3}\colon \mathrm{GL}(W_1)\times \mathrm{GL}(W_2)\times \mathrm{GL}(W_3)\times \mathrm{GL}(W_4)\times \mathrm{GL}(W_5)\times \mathrm{GL}(W_6)\times \mathrm{GL}(W_7)\ra \mathrm{GL}(W).\]
We then set \[w_{k_1,k_2,k_3}=w_{k_1,k_2,k_3,r,G(W)}\coloneq \iota_{k_1,k_2,k_3}(1_{W_1},\beta_{k_3}, \alpha_{a-k_1-k_3}, \alpha_{a-k_1-k_3},\alpha_{t-a},\alpha_{t-a},1_{W_7})\in G(W)\] if $G(W)$ is not metaplectic and if $G(W)=\m(W)$ we set \[w_{k_1,k_2,k_3}=w_{k_1,k_2,k_3,r,\m(W)}\coloneq (w_{k_1,k_2,k_3,r,\s(W)},1).\]

We will also define a second representative of the orbit parametrized by $k_1,k_2,k_3$ denoted by $v_{k_1,k_2,k_3}$ as follows. Let $b=\max(t,r)$ and decompose
\[W=W_1\oplus W_2'\oplus W_3'\oplus W_4'\oplus W_5',\, W_1=X_{k_1}\oplus Y_{k_1},\]\[ W_2'\coloneq X_{a-k_3,b}\oplus Y_{a-k_3,b},\]
\[W_3'\coloneq X_{k_1,a-k_3}\oplus X_{b, b+a-k_3-k_1},\, W_4'\coloneq Y_{k_1,a-k_3}\oplus Y_{b, b+a-k_3-k_1}, \]
\[W_5'\coloneq W_{b+a-k_3-k_1}.\]
As above we have an embedding
\[\iota_{k_1,k_2,k_3}'\colon \mathrm{GL}(W_1)\times \mathrm{GL}(W_2')\times \mathrm{GL}(W_3')\times \mathrm{GL}(W_4')\times \mathrm{GL}(W_5')\ra \mathrm{GL}(W).\] If $G(W)$ is not metaplectic we set
\begin{equation}\label{E:representative}v_{k_1,k_2,k_3}=v_{k_1,k_2,k_3,r,G(W)}\coloneq \iota_{k_1,k_2,k_3}'(1_{W_1},\gamma, \alpha_{a-k_1-k_3}, \alpha_{a-k_1-k_3},1_{W_5})\in G(W),\end{equation}
where $\gamma$ is the matrix
$\beta_{k_3}\in \ho(X_{r-k_3,r}\oplus Y_{t-k_3,t},X_{r-k_3,r}\oplus Y_{t-k_3,t})$ plus the matrix
\[\begin{cases}
    \alpha_{t-r}\in \ho(X_{r,t}\oplus Y_{r-k_3,t-k_3},  X_{r,t}\oplus Y_{r-k_3,t-k_3})&\text{ if } r\le t,\\
    \alpha_{r-t}\in \ho(X_{t-k_3,r-k_3}\oplus Y_{t,r},  X_{t-k_3,r-k_3}\oplus Y_{t,r})&\text{ if } r\ge t.
\end{cases}\]\
If $G(W)=\m(W)$ we again set $v_{k_1,k_2,k_3}=v_{k_1,k_2,k_3,r,\m(W)}\coloneq (v_{k_1,k_2,k_3,r,\s(W)},1)$.
\fakesubsection{}
The following lemma will be useful later on.
\begin{lemma}[{\cite[Lemma 5.2]{GanTakeda}}]\label{L:Ganquotient}
    Let $\xi$ be a character of $H_1$ such that $\xi^\lor\ncong \xi,\, \ain{r}{1}{q_W}$ and $\delta\in \ir(G(W_r))$ such that $r_{\overline{P_{(1)}}}(\delta)$ does not contain an irreducible subquotient of the form $\xi\otimes\delta'$.
    Then $\xi^r\rtimes\delta$ has a unique irreducible quotient denoted by $\delta_{\xi,r}$, $\xi^r\otimes\delta$ appears in $r_{\overline{P_{(r)}}}(\xi^r\rtimes\delta)$ with multiplicity $1$ and there exists no $\delta'\ncong\delta,\, \delta'\in\ir(G(W_r))$ such that $\xi^r\otimes'$ appears in $r_{\overline{P_{(r)}}}(\xi^r\rtimes\delta)$.
    
Moreover, if $\pi\in \ir(G(W))$ is such that $r_{\overline{P_{(1)}}}(\pi)$ does contain an irreducible subquotient of the form $\chi\otimes\delta^\lor$,
then there exist $r$ and $\delta\in\ir(G(W_r))$ such that $\pi\cong \delta_{\chi,r}$.
\end{lemma}
\fakesubsection{}
Let $G'$ be now either $H_n$ or $G(W)$.
We denote by
$S(G')$ the regular representation of $ G'\times G'$, \textit{i.e.} the space of Schwartz-functions on $G'$ with action \[(g_1,g_2)\cdot f\coloneq (h\mapsto f(g_1^{-{1}}hg_2)).\] If $G'$ is metaplectic we consider $\zeta_2$-equivariant Schwartz-functions:
See also \cite[3.II.3]{MVW} for the next lemma.
\begin{lemma}\label{L:regularrep}
For $\pi,\pi'\in \ir(G')$, the space \[\ho_{ G'\times  G'}(S( G'),\pi\otimes \pi')\neq 0\] if and only if  $\pi^\lor\cong \pi' $ and in this case it is $1$ dimensional.
More generally, if $\pi_1,\pi_2\in \rep( G')$ of finite length, then 
\[\dim_\C \ho_{ G'\times  G'}(S(G'), \pi_1\otimes \pi_2)=\dim_\C \ho_{G'}(\pi_1^\lor, \pi_2).\]
If $\chi$ is a character of $G'$, then there exists an isomorphism
\[S( G')\iso (\chi^\lor\otimes\chi)S( G').\]
\end{lemma}
\begin{proof}
    Let $C^\infty(G')=S(G')^\lor$ be the vector space of smooth functions on $G'$ on which $G'\times G'$ acts by left-right translation as on $S(G')$. Then \[\dim_\C \ho_{ G'\times  G'}(S(G'), \pi_1^\lor\otimes \pi_2^\lor)=\dim_\C \ho_{ G'\times  G'}(\pi_1\otimes \pi_2, C^\infty(G')).\]
Note now that $C^\infty(G')$ is the trivial representation of $\Delta G'$ non-compactly induced to $G'\times G'$. The claim then follows from Frobenius-reciprocity. The last isomorphism is given by sending a function $f\mapsto \chi f$.
\end{proof}

\begin{lemma}\label{L:reductionregularrep}
    Let $G'$ be as above and $P$ a parabolic subgroup of $G'$ with Levi-decomposition $P=M\ltimes N$.
    Then \[r_{P\times G}(S(G'))=\id_{M\times P}^{M\times G'}(S(M)) \]  
\end{lemma}
\begin{proof}
Since $S(G')=\id_{\Delta G'}^{G'\times G'}1$ we can apply the Geometric Lemma and note that $\Delta G'\bs G'\times G'$ is a single $P\times G'$-orbit.
\end{proof}
If $P$ is a maximal parabolic subgroup of $G'$ with Levi-component $M=M_1\times M_2$ and two representations of finite length of $M$ of the form $\pi_1\otimes \pi_2,\, \pi_1'\otimes  \pi_2'$ we recall
\[\ho_{M}(\pi_1\otimes\pi_2,\pi_1'\otimes \pi_2')\cong \ho_{M_1}(\pi_1\otimes\pi_1')\otimes  \ho_{M_2}(\pi_2\otimes\pi_2')\] as $\C$-vector spaces,
see for example \cite[Theorem 1.1]{Ra}.
\fakesubsection{}\label{S:minguezfiltration}
If $H_n\times H_m=\gl_n\times \gl_m$ let $\sigma_{n,m}$ be the space of Schwartz-functions on \[M_{n,m}\coloneq \mathrm{End}(E^m,E^n).\] It carries a natural action of $H_n\times H_m$ by 
\[(g_1,g_2)\cdot f\coloneq (x\mapsto f(g_1^{-1}xg_2)).\]
and admits a filtration $0=S_{t+1}\subseteq\ldots\subseteq S_0=\sigma_{n,m}$, $t=\min(n,m)$, where $S_i=S(M_{n,m}^i)$ denotes the space of Schwartz-functions on $M_{n,m}^i$, the linear maps of rank greater than or equal to $i$. The subquotients are of the form
\[\omega_i\coloneq S_{i+1}\bs S_{i}\cong S(M_{n,m,i})\cong \#-\id_{\overline{Q_{(n-i,i)}}\times {Q_{(m-i,i)}}}^{ H_n\times  H_m}({\underbrace{1}_{H_{n-i}}} \otimes {\underbrace{1}_{H_{m-i}}}\otimes S( H_{i})),\]
where $S(M_{n,m,i})$ denotes the space of Schwartz-functions on $M_{n,m,i}$, the linear maps of precisely rank $i$,
see \cite[3.II.2]{MVW}.
The morphism \[S(M_{n,m,i})\iso \#-\id_{\overline{Q_{(n-i,i)}}\times {Q_{(m-i,i)}}}^{ H_n\times  H_m}({\underbrace{1}_{H_{n-i}}} \otimes {\underbrace{1}_{H_{m-i}}}\otimes S( H_{i}))\] sends $f$ to $(g_1\times g_2)\mapsto f(g_1^{-1}1_i'g_2)$, where \[1_i'\coloneq \begin{pmatrix}
    0&0\\
    0&1_i
\end{pmatrix}.\]
If $H_n\times H_m=\widetilde{\gl_n}\times \widetilde{\gl_m}$, we write by abuse of notation also $\sigma_{n,m}$ for $\sigma_{n,m}^\psi$.
\begin{theorem}[{\cite[Theorem 1]{Minguez}}]\label{T:howedualII}
Let $n,m\in\NN,\, n\le m$ and $\pi\in \ir(\gl_n)$.
Then there exists $\pi'\in \ir(\gl_m)$, unique up to isomorphism, such that
\[\ho_{\gl_n\times \gl_m}(\sigma_{n,m},\pi\otimes \pi')\neq \{0\}.\]
Moreover, $\dim_\C \ho_{\gl_n\times \gl_m}(\sigma_{n,m},\pi\otimes \pi')=1$ and $\pi'$ is a quotient of 
\[\lvert-\lvert^{{\frac{m-2n-1}{2}}}\times\ldots\times \lvert-\lvert^{\frac{1-m}{2}}\times\lvert-\lvert^{\frac{m-n}{2}}\pi^\lor.\]
\end{theorem}\begin{rem}\label{R:meta}
    Note that the above theorem admits an obvious, but so far conjectural, generalization to the case where $H_n\times H_m$ is the metaplectic cover of $\gl_n\times \gl_m$, of which we unfortunately do not have a proof at the moment. It seems however possible that one could eventually adapt the argument of \cite{Minguez}, with the help of the results in \Cref{S:metaplectic}, to also cover this slightly more general case.
\end{rem}
We call an irreducible representation $\rho$ of $H_n$ square-irreducible if $\rho\times\rho$ is irreducible.
Note that for $\alpha$ a partition of $n$, the representation \[\underbrace{\lvert-\lvert^{{\frac{\alpha_1}{2}}}}_{H_{\alpha_1}}\times\ldots\times \underbrace{\lvert-\lvert^{{\frac{\alpha_k}{2}}}}_{H_{\alpha_k}} \cong \Z([1,\alpha_1]_{\lvert-\lvert^{-\frac{1}{2}}})\times\ldots\times \Z([1,\alpha_k]_{\lvert-\lvert^{-\frac{1}{2}}})\in \ir(H_n)\] is square-irreducible by \Cref{L:glirr2}.
\begin{lemma}\label{L:squareirreducible}
    Let $\alpha,k\in\mathbb{N},\,s\in \C$, $n=k\alpha$ and $\rho\coloneq \underbrace{\lvert-\lvert^{{\frac{s}{2}}}}_{H_{\alpha}},\,  \pi\coloneq \rho^k\in\ir(H_n)$. If $\tau$ is a smooth representation, not necessarily irreducible, such that either $r_{(n-\alpha,\alpha)}(\pi)\sra\tau\otimes\rho$ or $\rho\otimes \tau\hra r_{\overline{P_{(n-\alpha,\alpha)}}}(\pi)$ then $\tau\cong \rho^{k-1}$.
    \end{lemma}
\begin{proof}
We will only show the first claim, the second follows by duality. First of all, it is easy to see from the Geometric Lemma that each irreducible subquotient of $\tau$ is isomorphic to $\rho^{k-1}$ thus it suffices to show that $\tau$ is irreducible.
It is enough to assume $\tau$ is of length $2$. Moreover, since $\rho$ is square-irreducible, it was shown in the proof of \cite[Theorem 4.1.D]{lapid2022binary} that the intertwining operator $R_{\rho,\tau}\colon \rho\times\tau\ra\tau\times \rho$ has image isomorphic to $\pi$ and the so obtained map $\pi\hra\rho\times\tau$ is the map obtained by Frobenius-reciprocity from $r_{(n-\alpha,\alpha)}(\pi)\sra\tau\otimes\rho$. Now if $\tau_1$ is an irreducible subrepresentation of $\tau$ and $\tau_2=\tau/\tau_1$ the corresponding irreducible quotient, we obtain a commutative diagram \[\begin{tikzcd}
\rho\times\tau_1\arrow[r]\arrow[d,hookrightarrow]&\tau_1\times\rho\arrow[d,hookrightarrow]\\
    \rho\times\tau\arrow[r,"R_{\rho,\tau}"]\arrow[d,twoheadrightarrow]&\tau\times\rho\arrow[d,twoheadrightarrow]\\
    \rho\times\tau_2\arrow[r]&\tau_2\times\rho
\end{tikzcd}\]
where the top and bottom arrows are either $0$ or the intertwining operator by \cite[Lemma 2.3(1)]{MinLa18}. By above observation we know that the bottom arrow has to be non-zero and hence the order of the pole of $R_{\rho,\tau}$ and $R_{\rho,\tau_2}$ are equal, again by \cite[Lemma 2.3(1)]{MinLa18}. This implies that also the pole of $R_{\rho,\tau_1}$ is equal to the pole of $R_{\rho,\tau}$, since $\tau_1\cong\rho^{k-1}\cong\tau_2$ and hence also the top arrow is non-zero by \cite[Lemma 2.3(3)]{MinLa18}.
This however contradicts the fact that $R_{\rho,\tau}$ has irreducible image $\pi$.
\end{proof}

\fakesubsection{}
Next, we prove the two following corollaries of \Cref{T:howedualII}, therefore we assume $H_n=\mathrm{GL}_n$ in this section.
\begin{lemma}\label{C:supportzeta1}
    Let $n\le m\in \NN$, $\alpha=(\alpha_1,\ldots,\alpha_k)$ a partition of $n$, $p\in \ZZ$ and $\pi$ an irreducible representation of the form
    \[\pi\coloneq \underbrace{\lvert-\lvert^{{\frac{n-\alpha_1}{2}}+p}}_{H_{\alpha_1}}\times\ldots\times \underbrace{\lvert-\lvert^{{\frac{n-\alpha_k}{2}}+p}}_{H_{\alpha_k}}.\]
    Let $a\coloneq \max_i\alpha_i,\, b\coloneq \min_i\alpha_i$ and assume $p\ge a$ or $p<b-n$.
 Let $\pi'$ be the irreducible representation such that there exists a unique up to a scalar, non-zero morphism $f$
    \[f\colon \sigma_{n,m}\ra \pi\otimes\pi'.\]
    Then $f$ does not vanish on $S_{n}$ and there exists no morphism $\omega_l\ra \pi\otimes\pi'$ for $l<n$.
\end{lemma}
\begin{proof}
    It is enough to show that for $l<n$, there exists no morphism $\omega_l\ra \pi\otimes \pi'$.
    We note that
    \[\omega_l\cong \id_{\overline{Q_{(n-l,l)}}\times Q_{(m-l,l)}}^{H_n\times H_m}(\lvert-\lvert^{\frac{l}{2}} \otimes \lvert-\lvert^{-{\frac{l}{2}}}\otimes (\lvert-\lvert^{\frac{l-n}{2}}\otimes \lvert-\lvert^{\frac{m-l}{2}})S(H_l) ).\]
    and apply first \Cref{E:Bernstein} with respect to $\overline{Q_{n-l,n}}\times H_m$ and then \Cref{L:glred} together with the Geometric Lemma
to obtain a non-zero morphism
\[\lvert-\lvert^{\frac{l}{2}} \ra \lvert-\lvert^p(\underbrace{\lvert-\lvert^{-{\frac{\alpha_1-l_1}{2}}+{\frac{n-\alpha_1}{2}}}}_{H_{l_1}}\times\ldots \times \underbrace{\lvert-\lvert^{-{\frac{\alpha_k-l_k}{2}}+{\frac{n-\alpha_k}{2}}}}_{H_{l_k}})\]
for some partition $(l_1,\ldots, l_k)$ of $n-l$ with $l_i\le \alpha_i$.
By \Cref{L:glirr2} all but one $l_i$ are $0$ and there is one $i$ such that $l=n-l_i$. But this implies
\[{\frac{l}{2}}=-{\frac{\alpha_i-l_i}{2}}+{\frac{n-\alpha_i}{2}}+p\] and hence
\[n>l=n-\alpha_i+p\ge n-a+p\text{ and } 0\le l=n-\alpha_i+p\le n-b+p,\] a contradiction.
\end{proof}
\begin{lemma}\label{C:supportzeta2}
    Let $n\in \NN$, $\alpha=(\alpha_1,\ldots,\alpha_k)$ a partition of $n$ and $\pi$ the irreducible representation
    \[\pi\coloneq \underbrace{\lvert-\lvert^{{\frac{n-\alpha_1}{2}}}}_{H_{\alpha_1}}\times\ldots\times \underbrace{\lvert-\lvert^{{\frac{n-\alpha_k}{2}}}}_{H_{\alpha_k}}.\]
    Let $a\coloneq \max_i\alpha_i$
 and $f$ the, unique up to a scalar, non-zero morphism
    \[f\colon \sigma_{n,n}\ra \pi\otimes\pi^\lor.\]
    Then $f$ vanishes on $S_{n-a+1}$ and does not vanish on $S_{n-a}$.
\end{lemma}
Observe that by \Cref{L:glirr2} the representation $\pi$ of the corollary is indeed an irreducible representation.

Before we start with the proof, let us recall from \cite[3.III. 7]{MVW} how one constructs a morphism
\begin{equation}\label{E:zetaint}P\colon \sigma_{n,n}\ra\pi\otimes\pi^\lor\end{equation} for general irreducible representations $\pi$.
For $s\in \C$ let $\phi_s$ be a matrix coefficient of $\pi\lvert-\lvert ^s$, \textit{i.e.} a linear combination of maps of the form \[g\mapsto v^\lor(\pi(g)\lvert g\lvert ^s v)\] for $v\in \pi,\, v^\lor \in \pi^\lor$, and let $f$ be an element of $\sigma_{n,n}$. Let $\mathrm{d} g$ be a Haar measure on $H_n$ and define for $\mathrm{Re}\,s>>0$ the Godement-Jaqcuet zeta integral
\[P(s,f,\phi_s)\coloneq{\frac{1}{L(s-{\frac{n-1}{2}},\pi)}}\int_{H_n}f(g)\phi_s(g)\,\mathrm{d} g,\]
where $L(s,\pi)$ is the standard $L$-function of $\pi$, see for example \cite{godement1972zeta}.
Then $P(s,f,\phi_s)$ is a polynomial in $q^s$ and $q^{-s}$ and can be analytically continued to $s=0$. By specifying $s=0$, one obtains a morphism $\sigma_{n,n}\ra \pi\otimes\pi^\lor$.

\begin{proof}[Proof of \Cref{C:supportzeta2}]
We first show that if there exists a non-zero morphism $\omega_l\ra \pi\otimes\pi^\lor$ then $l=n$ or $\ain{l}{n-\alpha_1}{n-\alpha_k}$.
As in \Cref{C:supportzeta1} recall that if $l<n$ \[\omega_l\cong \id_{\overline{Q_{(n-l,l)}}\times Q_{(n-l,l)}}^{H_n\times H_n}(\lvert-\lvert^{\frac{l}{2}} \otimes \lvert-\lvert^{-{\frac{l}{2}}}\otimes (\lvert-\lvert^{\frac{l-n}{2}}\otimes \lvert-\lvert^{\frac{n-l}{2}})S(H_l) ).\]
Applying first \Cref{E:Bernstein} with respect to $\overline{Q_{n-l,l}}\times H_n$ and then \Cref{L:glred} together with the Geometric Lemma
gives a non-zero morphism
\[\lvert-\lvert^{\frac{l}{2}} \ra \underbrace{\lvert-\lvert^{-{\frac{\alpha_1-l_1}{2}}+{\frac{n-\alpha_1}{2}}}}_{H_{l_1}}\times\ldots \times \underbrace{\lvert-\lvert^{-{\frac{\alpha_k-l_k}{2}}+{\frac{n-\alpha_k}{2}}}}_{H_{l_k}}\]
for some partition $(l_1,\ldots, l_k)$ of $n-l$ with $l_i\le \alpha_i$.
By \Cref{L:glred} it follows that all but one $l_i$ are $0$ and hence for this one non-zero $l_i=n-l$ we have \[{\frac{l}{2}}=-{\frac{\alpha_i-l_i}{2}}+{\frac{n-\alpha_i}{2}}\] which implies $l=n-\alpha_i$.

In order to show the lemma, it thus suffices by \Cref{T:howedualII} to show that for the map $P$ of \Cref{E:zetaint}, $P$ vanishes on $S_{n-\alpha_i}$ for $n-\alpha_i>n-\alpha$ and $S_n$. Since $\pi$ is an irreducible induced representation, we can assume without loss of generality that $\alpha_1\le\ldots \le \alpha_k=a$.
We start with the case $k=1$ and hence $\pi$ is the trivial representation. 
Then $L(s-{\frac{n-1}{2}},\pi)$ has poles at $s=i,\,\ain{i}{0}{n-1}$ and therefore ${L(s-{\frac{n-1}{2}},\pi)}^{-1}$ vanishes at $0$. Since for $f\in S_n\hra \sigma_{n,n}$ the integral 
$\int_{H_n}f(g)\phi_0(g)\,\mathrm{d} g$ converges and ${L(s-{\frac{n-1}{2}},\pi)}^{-1}$ vanishes at $0$, $P(0,f,\phi)$ vanishes for all $f\in S_n$. This finishes the case $k=1$.

If $k>1$, we can use \cite[Proposition 2.3]{JL2}, which shows that for fixed $f\in \sigma_{n,n}$ and matrix coefficient $\phi$ of $\pi$ we can write $P(0,f,\phi)$ as the finite linear combination of functionals of the form
\[\prod_{i=1}^kP(0,f_i,\phi_i),\] for $f_i\in \sigma_{\alpha_i,\alpha_i}$ and $\phi_i$ a matrix coefficient of the trivial character of $H_{\alpha_i}$. Moreover, in the proof of 
\cite[Proposition 2.3]{JL2} the author shows that 
\[\prod_{i=1}^kP(0,f_i,\phi_i)=P(0,f',\phi')\] for $\phi'$ a matrix coefficient of $\pi$ and $f'\in \sigma_{n,n}$ satisfying 
\[\int_{U_\alpha}f(xmu)\,\mathrm{d} u=f_1(x_1m_1)\cdot\ldots\cdot f_k(x_km_k)\] for all $m=(m_1,\ldots m_k)$ in the Levi-component of $Q_\alpha$, $x_i\in M_{\alpha_i,\alpha_i}$,
\[x\coloneq \begin{pmatrix}
    x_1&0&0\\
    0&\ddots&0\\
    0&0&x_k
\end{pmatrix}\]and a suitable Haar-measure $\mathrm{d}u$ on the unipotent part $U_\alpha$ of $Q_\alpha$. In particular, if we assume that $P(0,f',\phi')$ does not vanish for $f'\in S_k$ for $k>n-a$, 
we know from the case $k=1$ that each of the $f_i$ does not vanish on $0$. Therefore we obtain from the specific form of $f'$ that it does not vanish on some element of rank at most \[\max_{u\in U_\alpha}\mathrm{rank}(u)=n-a.\] This contradicts $f'\in S_k$.
\end{proof}
\section{Filtrations}\label{S:filtration}
In this section, we study filtrations of $I_{W,W}(\chi,s)$ and its Jacquet module along a maximal parabolic subgroup. The first filtration was already introduced in \cite{KudlaRallis}.
\fakesubsection{}
Consider the space $W\oplus W$ where the second copy of $W$ is equipped with the inner product $-\langle.,.\rangle$.
This induces a natural morphism \[\iota\colon G(W)\times G(W)\rightarrow G(W\oplus W),\]
which is an embedding except when $G(W)=\m(W)$. In this case it restricts to an embedding \[\iota\colon \s(W)\times \s(W)\hookrightarrow \s(W\oplus W)\]
and we choose it so that $(\zeta,\zeta')\mapsto (\zeta\zeta')$ on the $\zeta_2$-part. We write from now on $W^i$, $X_{a}^i$, $X_{a,b}^i, Y_{a}^i$, $Y_{a,b}^i$ for the $i$-th copy of the respective space in $W\oplus W$ for suitable $a,b$. 
Let $\chi$ be a unitary character of $H_n$, $s\in \C$
and $Y\coloneq X_{q_W}^1\oplus X_{q_W}^2\oplus \Delta W_{q_W}\subseteq W\oplus W$ and decompose the corresponding parabolic subgroup as $P(Y)=M(Y)\ltimes N(Y)$.
We consider the induced representation, called the degenerate principal series representation of $G(W)\times G(W)$
\[I_{W,W}(\chi,s)\coloneq \iota^*(\id_{P(Y)}^{G(W\oplus W)}(\chi\lvert-\lvert^s)).\] This representation will preoccupy us throughout the remaining paper.
Set \[\L_W\coloneq P(Y) \bs G(W\oplus W)\] to be the space of $n$-dimensional isotropic subspaces of $W\oplus W$.  
The above embedding gives a $G(W)\times G(W)$-action on $\L_W$ with locally-closed orbits $\Omega_t$ for $\ain{t}{0}{q_W}$ given by \[\Omega_t\coloneq \{U:\dim_E (U\cap W^1)= \dim_E (U\cap W^2)=t\}.\]
Set $\Omega^t\coloneq \bigcup_{t'\le t}\Omega_{t'}$, which is an open subset of $\L_W$.
By \Cref{P:shortexact} we have a filtration
\[0=I_{-1}\subseteq I_0\subseteq\ldots\subseteq I_{q_W}=I_{W,W}(\chi,s),\]
where \[I_t\coloneq \{f\in I_{W,W}(\chi,s):\restr{f}{\overline{\Omega_{t+1}}}=0 \}.\]
Define 
\[\sigma_t\coloneq I_{t-1}\bs I_{t}.\]
The following theorem was already proved in \cite[§1]{KudlaRallis}. To make later arguments clearer, we write it out.
\begin{lemma}[{\cite[§1]{KudlaRallis}}]\label{L:KRfiltration}
The subquotients $\sigma_t\coloneq I_{t-1}\bs I_{t}$ are of the form
\[\sigma_t\cong (1\otimes \restr{\chi}{\Z(G(W)})\id_{P_{(t)}\times P_{(t)} }^{G(W)\times G(W)} ({\underbrace{\chi\lvert-\lvert}_{H_t}}^{s+{{\frac{t}{2}}}}\otimes {\underbrace{\chi\lvert-\lvert}_{H_t}}^{s+{\frac{t}{2}}}\otimes S(G(W_t))).\]
\end{lemma}
\begin{proof}
We introduce the following notation. If $A,\,B,\,C,\,D$ are vector spaces over $E$, 
\[a\in \ho(A,C),\, b\in \ho(B,C),\, c\in \ho(A,D),\, d\in \ho(B,D),\]
we denote by \[\begin{pmatrix}
    a&b\\
    c&d\\
\end{pmatrix}\in \ho(A\oplus B,C\oplus D)\] the corresponding linear map.
Pick $ \delta_t\in G(W\oplus W)$ with ${P(Y) \delta_t\in \Omega_t}$ as follows.
Recall $\alpha_{q_W-t},\beta_t$ of \Cref{S:alphabeta} and for $V\coloneq (X_{t,q_W}^1\oplus Y_{t,q_W}^1)\oplus (X_{t,q_W}^2\oplus Y_{t,q_W}^2)$ let
\[ x_{0,V}\coloneq
    \begin{pmatrix}
    0&\alpha_{q_W-t}\\
    \beta_t&1_{2q_W-2t}\\
\end{pmatrix}\in \ho((X_{t,q_W}^1\oplus Y_{t,q_W}^2)\oplus (X_{t,q_W}^2\oplus Y_{t,q_W}^1),V).
\]
We now
set $V'\coloneq X_{t}^1\oplus Y_{t,1}\oplus X_{t}^2\oplus Y_{t}^2\oplus W_{q_W}^1\oplus W_{q_W}^2$ and define $\delta_t$ to be the image of $  1_{V'}\times x_{0,V}$ (respectively $(1_{V'},1)\times x_{0,V} $ in the metaplectic case) in $G(W\oplus W)$ under the natural morphism
\[ G(V')\times G(V)\ra G(W\oplus W).\]
In the metaplectic case we again take the morphism respecting the multiplication.
The stabilizer of $P(Y) \delta_t$ in $\L_W$ under the action of $G(W)\times G(W)$ is the subgroup
\begin{equation}\label{E:stab}R_t\cong  H_t\times H_t\times \Delta G(W_t)\ltimes (N_t\times N_t)\subseteq P_{(t)}\times P_{(t)},\end{equation}
where $N_t$ denotes the unipotent component of $P_{(t)}$.
Next, let us recast the representation $I_{W,W}(\chi,s)$ in the language of $\ell$-sheaves by sending the representation $\chi\lvert-\lvert^s\delta_{P(Y)}^{\frac{1}{2}}$ through the following diagram.
\[\begin{tikzcd}
\rep(P(Y))\arrow[r,"\id"]&\sh(\L_W,G(W\oplus W))\arrow[r,"\rs"]&\sh(\L_W,G(W)\times G(W))\arrow[d,"\scc"]\\\rep(H_1)\arrow[u]&&\rep(G(W)\times G(W))\\
\end{tikzcd}\]
By \Cref{P:sheafin} this computes precisely $I_{W,W}(\chi,s)$.
Using this, we denote by $\scf^{s,\chi}$ the sheaf in $\sh(\L_W,G(W)\times G(W))$ such that $\scc(\scf^{s,\chi})=I_{W,W}(\chi,s)$.
The filtration of $\L_W$ by $\Omega^t$ gives by \Cref{P:shortexact} a filtration of $\scc(\scf^{s,\chi})$ via the short exact sequence
\[\begin{tikzcd}\sh(\Omega^{t-1},G(W)\times G(W))\arrow[r, ]\arrow[d,maps to]&\sh(\Omega^{t},G(W)\times G(W))\arrow[r]\arrow[d,maps to]&\sh({\Omega_t},G(W)\times G(W))\arrow[d,maps to]\\
\scf^{s,\chi}_c(\Omega^{t-1})\arrow[r,hookrightarrow]& \scf^{s,\chi}_c(\Omega^{t})\arrow[r,twoheadrightarrow]&\sigma_t\\
\end{tikzcd}\]
But $\sh(\Omega_t,G(W)\times G(W))\xrightarrow{\scc}\rep(G(W)\times G(W))$ is by \Cref{P:sheafin} equivalent to the following composition. 
\[\begin{tikzcd}
    \sh(\Omega_t,G(W)\times G(W))\arrow[r, "\rs"]\arrow[d,"\scc"]&\rep(R_t)\arrow[d,"\id"]\\ \rep(G(W)\times G(W))&\sh(\Omega_t,G(W)\times G(W))\arrow[l, "\scc"]\\
\end{tikzcd}\]
This gives the claimed induced representation.
Namely, for $g=(m_1,m_2,g,g,n_1,n_2)\in  R_t$, we have
\[\delta_tg\delta_t^{-1}=(\underbrace{m}_{H(Y)},n)\in P(Y)\] and by an easy calculation $\lvert m\lvert=\lvert m_1 m_2\lvert$.
Then $H_t \times H_t$ acts by $\chi \lvert-\lvert^{s}\delta_{P(Y)}^{\frac{1}{2}}$ on $\rs(\scf^{s,\chi})$,  $\Delta G(W_t)$ acts by $\restr{\chi}{\Z(G(W_t))}$ and $N_t\times N_t$ acts trivially. Moreover,
\[\rep(R_t)\xrightarrow{\id}\sh(\Omega_t,G(W)\times G(W))\] factors as
\[\rep(R_t)\xrightarrow{\id}\sh(R_t\bs  P_{(t)}\times P_{(t)})\xrightarrow{\id}\sh(\Omega_t,G(W)\times G(W)).\]
Since the image of a character $\chi$ under 
\[\rep(\Delta  G(W_t))\xrightarrow{\id} \sh(\Delta G(W_t)\bs  G(W_t)\times  G(W_t), G(W_t)\times  G(W_t))\xrightarrow{\scc} \rep( G(W_t)\times G(W_t))\] is the regular representation twisted by $(1\otimes \chi^{-1})$, the final assertion follows by noting that $\scc\circ\id$ gives $\#-\id$ and writing $\mathrm{1}=\delta_{P_{(t)}\times P_{{(t)}}}^{\frac{1}{2}} \delta_{P_{(t)}\times P_{(t)}}^{-{\frac{1}{2}}}$ to normalize the induction.
\end{proof}
The isomorphism \[\scf_c^{s,\chi}(\Omega_t)\iso (1\otimes \restr{\chi}{\Z(G(W)})\id_{P_{(t)}\times P_{(t)} }^{G(W)\times G(W)} ({\underbrace{\chi\lvert-\lvert}_{H_t}}^{s+{{\frac{t}{2}}}}\otimes {\underbrace{\chi\lvert-\lvert}_{H_t}}^{s+{\frac{t}{2}}}\otimes S(G(W_t))).\]
can therefore be written out explicitly as
\[f\mapsto (\underbrace{(g_1\times g_2)}_{G(W)\times G(W)}\mapsto (\underbrace{h}_{G(W_j)}\mapsto f(\delta_t\iota(g_1,hg_2)))) .\]
 An irreducible representation $\pi\in \ir(G(W))$ is said \emph{to appear on the boundary component} if there exists a non-zero morphism $f\colon I_{W,W}(\chi,s)\ra \pi\otimes\pi^\lor$ which vanishes on $I_{t-1}$ and does not vanish on $I_t$ for some $t>0$. Note that in this case there exists then a non-zero morphism $\sigma_t\ra\pi\otimes\pi^\lor$.
 \begin{rem}
     Note that one could a prioi hope that one gets a similar statement for $I_{W,W}(\chi,s)$ as for $\sigma_{n,n}$, in the sense that if there exists a non-zero morphism $I_{W,W}(\chi,s)\ra \pi\otimes\pi'$, then this implies $\pi'\cong\chi\pi^\lor$. However, if for example $W$ is symplectic, $\sigma_{q_W}$ does not have to be cosocle irreducible, indeed for $s=-\frac{q_W}{2}$, $\sigma_{q_W}$ is semisimple of length $4$.
 \end{rem}
\fakesubsection{}
Let $\ain{r}{1}{q_W}$, $P=P_{(r)}\times G(W)\subseteq G(W)\times G(W)$ be a standard parabolic subgroup and define the following subsets of $\L_W$ for $\ain{k}{1}{r},\, \ain{j}{r}{q_W}$.
\[\Gamma^{k,j}\coloneq \{U:\dim_E (U\cap X_{r}^1)\le k,\, \dim_E( p_{Y_{r}^1}(U))\ge r-k,\]\[ \dim_E(p_{X_{r}^1\bs W^1}(U))\ge 2q_W-j-k\},\]
where $p_{Y_{r}^1}$ respectively $p_{X_{r}^1\bs W^1}$ is the projection to $Y_{r}^1$ respectively $X_{r}^1\bs W^1$. This subspace is $P$-invariant and $\Gamma^{k,j}$ is open since $\dim_E(\cdot \cap X)$ is upper semicontinous and $\dim_E(p_X(\cdot))$ is lower semicontinous for a suitable vector space $X$.
We write $(k,j)\le (k',j')$ if $k\le k',\, j\le j'$. Then $\Gamma^{k,j}\subseteq \Gamma^{k',j'}$ if and only if $(k,j)\le (k',j')$. Define
\[\Gamma_{k,j}\coloneq \Gamma^{k,j}\backslash \bigcup_{(k',j')<(k,j)}\Gamma^{k',j'}=\{U:\dim_E (U\cap X_{r}^1)= k,\, \dim_E (p_{Y_{r}^1}(U))= r-k,\]\[ \dim_E(p_{X_{r}^1\bs W^1}(U))= 2q_W-j-k\}.\]
It is then clear that $\bigcup_{k,j}\Gamma_{k,j}=\L_W$ and the union is disjoint. Hence this gives a stratification of $\L_W$ by $P$-invariant locally closed subspaces of $\L_W$. Recall the parabolic subgroup $P'_{(j-r)}$ defined at the end of \Cref{S:parabolicsp}.
\begin{theorem}\label{T:filtrationreduction}
    The representation $r_{P}(I_{W,W}(\chi,s))$ has a filtration with subquotients $\tau_{k,j},\, \ain{k}{1}{r},\, \ain{j}{r}{q_W}$ such that there exists isomorphisms
    \[A_{k,j}\colon \tau_{k,j}\iso (1\otimes\chi)\id_{Q_{(k,r-k)}\times P_{(j-r)}' \times P_{(j)}}^{H_r\times G(W_r)\times G(W)}(\underbrace{\chi\lvert-\lvert^{s+{\frac{k}{2}}}}_{H_k}\otimes 
    \]\[\otimes \sigma_{r-k,j}(\chi\lvert-\lvert^{-s-j+{\frac{r-k}{2}}}\otimes \chi\lvert-\lvert^{s+{\frac{j}{2}}})\otimes \underbrace{\chi\lvert-\lvert^{s+k+{\frac{j-r}{2}}}}_{H_{j-r}}\otimes S(G(W_j))),\]
    where $P_{(j-r)}$ is a parabolic subgroup of $G(W_r)$.
\end{theorem}
\begin{proof}
We write the Levi-decomposition of $P$ as $P=M\ltimes N$. Let $\scf^{s,\chi}$ be the sheaf in $\sh(\L_W,G(W)\times G(W))$ such that $\scf^{s,\chi}_c(\L_W)=I_{W,W}(\chi,s)$ as in the proof of \Cref{L:KRfiltration}.
 Observe that 
$\Gamma_{k,j}$ is covered by 
\[\Gamma_{k,j,t}\coloneq \Omega_t\cap \Gamma_{k,j}=x_t w_{k,j-r,t+r-j-k}\iota(P),\,\ain{t}{j-r+k}{j}.\]
To see this it is enough to note that for $U=x_t w_{a,b,t-a-b}\iota(p),$ $a,b\in\mathbb{N},\, a+b\le t,\, p\in P$
\[\dim_E (U\cap X_{r}^1)=a,\, \dim_E(p_{X_{r}^1\bs W^1}(U))=2q_W-r-b-a,\]\[ \dim_E(p_{Y_{r}^1}(U))=r-a.\]
We now set \[\#-\tau_{k,j}\coloneq\scf_c^{s,\chi}({\Gamma_{k,j}})_{N},\,  \tau_{k,j}\coloneq \delta_P^{-{\frac{1}{2}}}\#-\tau_{k,j}\] and since the $\Gamma_{k,j}$ cover $\L_W$ and are locally closed, we obtain a filtration of $r_P(I_{W,W}(\chi,s))$ with subquotients $\tau_{k,j}$'s by \Cref{P:shortexact}. Next, we will define a morphism
\[A_{k,j}'\colon \#-\tau_{k,j}\ra  (1\otimes\chi)\#-\id_{Q_{k,r-k}\times P_{(j-r)}' \times P_{(j)}}^{H_r\times G(W_r)\times G(W)}(\underbrace{\chi\lvert-\lvert^s\delta_{P(Y)}^{{\frac{1}{2}}}}_{H_k}\otimes 
    \]\[\otimes \sigma_{r-k,j}(\chi\lvert-\lvert^{-s}\delta_{P(Y)}^{-{\frac{1}{2}}}\lvert-\lvert^{n-j-k-\eta} \otimes \chi\lvert-\lvert^s\delta_{P(Y)}^{{\frac{1}{2}}})\otimes \underbrace{\chi\lvert^{s-r+k}\delta_{P(Y)}^{\frac{1}{2}}}_{H_{j-r}}\otimes S(G(W_j))).\]
To do so we first define a morphism
$\phi:M_{r-k,j}\times G(W_j)\ra G(W\oplus W)$.
This is done as follows. Let $h=(x,g)$ be an element of the left side and let $l$ be the rank of $x$ and set $t=j-l$. We then write \begin{equation}\label{E:E5}x=m_2^{-1}1_{l}'tm_1\end{equation} with $m_1\in H_j,\, m_2\in H_{r-k}$ and \[1_l'\coloneq \begin{pmatrix}
    0&0\\
    0&1_l
\end{pmatrix}\in M_{r-k,j}.\]
Next let $\iota_1\colon H_j\ra G(W)$ be the multiplicative morphism sending $m_1\in H_j$ to the element $(m_j, 1_{W_{j}})$ in the Levi-subgroup of the standard parabolic $P_{(r)}.$ Similarly, we define $\iota_2\colon H_{r-k}\ra G(W)$ by sending first \[m_2\mapsto \begin{pmatrix}
    1_k&0\\
    0&m_2
\end{pmatrix}\]
and then embedding this element into the Levi subgroup of the standard parabolic $P_{(r)}$ of $G(W)$. Finally, let $\iota_3\colon G(W_j)\ra G(W)$ by sending $g$ to the element $(1_j,g)$ in the Levi subgroup of the standard parabolic subgroup $P_{(j)}$ of $G(W)$.
We then set \[\phi(h)\coloneq \iota(\iota_2(\mathfrak{c}({}^tm_2)),\iota_1(m_1^{-1}))\delta_{t}\iota(w_{k,j-r,t+r-j-k},1_W)\iota(\iota_2(m_2),\iota_1(m_1)\iota_3(g)),\]
where $w_{k,j-r,t+r-j-k}=w_{k,j-r,t+r-j-k,r}$ is the element we introduced in \Cref{S:geolemsp}.
We need to show that this is independent of the choice of $m_1$ and $m_2$. Recall that the stabilizer of $1_l'$ under the action of $H_{r-k}\times H_j$ in the sense of \Cref{E:E5} is of the form \[p'=\begin{pmatrix}
p_3&0\\
p_2'&p_1'\\
\end{pmatrix}\in H_{r-k},\,p=\begin{pmatrix}
p_1&p_2\\
0&p_3\\
\end{pmatrix}\in H_j ,\]
where $p_1\in H_{t},\, p_2\in M_{t,l},\,p_3\in H_l,\, p_1'\in H_{r-k-l},\,p_2'\in M_{l,r-k-l}.$
A straightforward calculation shows then that if $p_3=1_l$
\[\iota(1_W,\iota_1(p^{-1}))\delta_{t}\iota(w_{k,j-r,t+r-j+k},1_W)\iota(1_W,\iota_1(p))=\delta_{t}\iota(w_{k,j-r,t+r-j+k},1_W)\]
and
\[\iota(\iota_2(\mathfrak{c}({}^tp')),1_W)\delta_{t}\iota(w_{k,j-r,t+r-j+k},1_W)\iota(\iota_2(p'),1_W)=\delta_{t}\iota(w_{k,j-r,t+r-j+k},1_W).\]
Moreover, if $p_1=1_t,\, p_1'=1_{r-k-l},\, p_2=0,\, p_2'=0$ one has the identity
\[\iota(\iota_2(\mathfrak{c}({}^tp')),\iota_1(p^{-1}))\delta_{t}\iota(w_{k,j-r,t+r-j+k},1_W)=\delta_{t}\iota(w_{k,j-r,t+r-j+k},1_W)\iota(\iota_2(p'^{-1}),\iota_1(p^{-1})).\]
Combining these three equalities gives $\phi({p'}^{-1}1_lp,1_W)=\phi(1_l,1_W)$ and hence shows that $\phi$ is well defined. Note moreover, that if we restrict the morphisms to those elements of a fixed rank $l$, it is continuous and
for $a=(m,g)\in H_{r-k}\times G(W_j), b=(m',g')\in H_j\times G(W_j)$
\begin{equation}\label{E:equi}
\phi(a^{-1}hb)=\iota(\iota_2(\mathfrak{c}({}^tm)),\iota_1(m'^{-1})\iota_3(g^{-1}))\phi(h)\iota(\iota_2(m)\iota_3(g'),\iota_1(m')).
\end{equation}
We then define $A_{k,j}'$ as the morphism sending a section $f\in \scf_c^{s,\chi}(\Gamma_{k,j}), \, f\colon P(Y)\Gamma_{k,j}\ra \C$ on the left hand side to
\[\underbrace{g}_{G(W)}\times \underbrace{m}_{H_r\times G(W_r)}\mapsto\left(\underbrace{h}_{M_{r-k,j}\times G(W_j)} \mapsto \int_{N_{k,j}\bs N}f(\phi(h)\iota(nm,g))\, \mathrm{d}n\right),\]
where we fix the choice of a Haar-measure on $N$ and $N_{k,j}\coloneq N_{k,j,j}$ with
\[N_{k,j,t}\coloneq \iota^{-1}(\iota(N,1_W)\cap \iota(w_{k,j-r,t+r-j-k}^{-1},1_W)\delta_t^{-1}N(Y)\delta_t\iota(w_{k,j-r,t+r-j-k},1_W))\] for $\ain{t}{j-r+k}{j}$. 
Putting now all issues of well-definedness aside for a moment, it follows from the properties of a Haar-measure that this morphism factors through $\scf_c^{s,\chi}(\Gamma_{k,j})_N$.

Several things need now to be checked in order for this morphism to well defined.
We start with the following lemma.
\begin{lemma}\label{L:stabilizerofunipotent}
For all $\ain{t,t'}{j-r+k}{j}$
    \[N_{k,j,t}=N_{k,j,t'}.\]
\end{lemma}
\begin{proof}
Recall from the proof of \Cref{L:KRfiltration} that the stabilizer of $x_{t}=P(Y)\delta_t$ under the action of $G(W)\times G(W)$ is equal to \[R_t\cong H_t\times H_t\times \Delta G(W_t)\ltimes (N_t\times N_t).\] Here we denote now by $N_t=N(X_t)$ again the unipotent part of the parabolic subgroup $P_{(t)}=P(X_t)$ in $G(W)$ and we write from now on $w(H)\coloneq w^{-1}Hw$ for any closed subgroup $H$ of $G(W)$. 
Thus $N_{k,j,t}$ is \[
    N_{k,j,t}=N_{r}\cap w_{k,j-r,t+r-j-k}(N_t).
\]
    Note the following equalities. Let $\ain{a,b}{0}{q_W-1}$. Firstly, $N(X_a)\cap N(X_{b,q_W})$ consists of those elements in $N(X_a)$ which are the identity except on $Y_{0,q_W}$ and induce the $0$-map in $\ho(Y_b,X_b).$
    Moreover, \[N(X_{a,b})\cap N(Y_{a,b})=\{1\}\] and if $a\le b$ $N(X_a)\cap N(X_b)$ consists of those elements in $N(X_a)$ which are the identity on $X_{a,b}\oplus Y_{a,b}$.
    Using this, the claim follows easily. 
    Indeed, note that 
    \[w_{k,j-r,t+r-j-k}(N_t)=N(X_k\oplus Y_{k,t-j+r}\oplus X_{q_W-j+r,q_W} ).\]
It follows that therefore $N_r\cap w_{k,j-r,t+r-j-k}( N_t)$ consists of the following elements of $N_r$. If $j-r\neq 0, k\neq r$ they are the identity except on $X_k\oplus Y_k\oplus Y_{q_W-j+r,q_W}$ and induce the $0$-map in $\mathrm{Hom}(Y_{q_W-j+r}, X_{q_W-j+r})$ . If $j=r, k\neq r$, it consists of those elements in $N_r$ which are the identity everywhere except on $X_k\oplus Y_k$. Finally,
if $k=r$ then $j=t$ and it consists of those elements in $N_r$ which are the identity on $X_{r,j}\oplus Y_{r,j}$. This description is independent of $t$.
\end{proof}
It implies that $n\mapsto f(\phi(h)\iota(nm,g))$ is invariant under $N_{k,j}$ and hence the integral makes sense. Let us check now that the integral converges. Indeed, let $K$ be the compact support of $f$ on $\L_W=P(Y)\bs G(W\oplus W)$. From the definition of $N_{k,j}$ it follows that $x(N_{k,j}\bs N)$ defines a closed subspace of $\L_W$ for all $x\in \Gamma_{k,j}$ and hence $xK\cap xN_{k,j}\bs N$ is again compact and therefore we have convergence. Next, it is easy to see that $A_{k,j}'(f)$ is also compactly supported with respect to $h$ and $A_{k,j}'(f)$ is locally constant with respect to $g$ and $m$. A priori it is however only given that $A_{k,j}'(f)$ is locally constant with respect to $h$ when we restrict it to elements $h$ whose rank in $M_{r-k,j}$ is $l$ for some fixed $l$.
Indeed, for $a=(a',1_{W_j})\in H_{r-k}\times G(W_j),\,b\in H_j\times G(W_j)$ both close to the identity, we have thanks to the $P(Y)$-invariance, $f$ being locally constant, \Cref{E:equi} and a change of variables 
\[\int_{N_{k,j}\bs N_r}f(\phi(h)\iota(nm,g))\,\mathrm{d}n=\int_{N_{k,j}\bs N_r}f(\phi(b^{-1}ha)\iota(nm,g))\,\mathrm{d}n.\]
To show that $A_{k,j}'(f)$ is locally constant with respect to $h$ without this restriction to a fixed rank, we argue as follows. Fix a rank $l$ and let $\ain{i}{l+1}{r-k}$ and for $e \in F$ \[m_e \coloneq \begin{pmatrix}
    0&0&0\\
    0&e 1_{i-l}&0\\
    0&0&1_l
\end{pmatrix}= 1_i'n'_e\] with \[n'_e \coloneq \begin{pmatrix}
    1_{j-i}&0&0\\
    0&e 1_{i-l}&0\\
    0&0&1_l
\end{pmatrix}.\]
Note then that the centralizer of $w_{k,j-r,r-k-l}^{-1}\delta_{j-l}^{-1}\delta_{j-i}w_{k,j-r,r-k-i}$ contains $\iota(H_r\times N_r,1_W)$, as it is the identity plus a morphism in $\mathrm{Hom}(X,Y)$ represented by the matrix $\beta_{i-l}$,
where \[
    X\coloneq Y_{l+k,i+k}^1\oplus X_{j-i,j-l}^2,\, Y\coloneq   Y_{j-i,j-l}^2\oplus X_{l+k,i+k}^1.\]
Thus for all $p\in H_r\ltimes N_r$
\[\iota(p^{-1},1_W)\phi(1_l',1_{W_j})^{-1}\phi(m_e,1_{W_j})\iota(p,1_W)=\]\[=\iota(1,\iota_1({n'}_e^{-1}))\phi(1_l',1_{W_j})^{-1}\phi(1_i',1_{W_j})\iota(1_W,\iota_1(n'_e)).\]
We can describe the last element explicitly as $1_{W\oplus W}$ plus a morphism in $\mathrm{Hom}(X,Y)$ represented by the matrix
$e\beta_{i-l}$
and for any open neighborhood of the identity we can choose $e$ such that the above element is contained in it.
Using that $f$ is locally constant  and a change of variables we thus obtain that that \[\int_{N_{k,j}\bs N}f(\phi(1_l',1_{W_j})\iota(nm,g))\,\mathrm{d}n=\]\[=\int_{N_{k,j}\bs N}f(\phi(1_l',1_{W_j})\iota(nm,g)\phi(1_l',1_{W_j})^{-1}\iota(n^{-1},1_W)\phi(m_e,1_{W_j})\iota(n,1_W))\,\mathrm{d}n=\]\[=\int_{N_{k,j}\bs N}f(\phi(m_e,1_{W_j})\iota(nm,g))\,\mathrm{d}n\] and hence $A_{k,j}'(f)(g,m)(1_l,1_{W_j})=A_{k,j}'(f)(g,m)(m_e,1_{W_j})$ for all $g$ and $m$.
 If $h$ is another element in $M_{r-k,j}$ of rank $i$ different from $(m_e,1_{W_j})$ such that $h$ is close to $1_l'$
we can use that $f$ restricted to a fixed rank is locally constant together with the fact that $\phi$ is continuous on those elements to show that $A_{k,j}(f)$ is locally constant with respect to $h$.

Next, we need to check that $A_{k,j}'(f)$ behaves under left translation with respect to $g$ and $m$ as is required by parabolic induction. For the invariance regarding the unipotent part of $P_{(j)},\, Q_{(k,r-k)}$ and $P_{(j-r)}'$ we argue as follows. Firstly, for $n$ in the unipotent part of $Q_{(k,r-k)}$ or $P_{(j-r)}'$, we have that for $\ain{t}{j-r+k}{j}$,  $w_{k,j-r,t-j-k+r}nw_{k,j-r,t-j-k+r}^{-1}\in \iota (N_t,1_W)$ and hence the invariance follows from \Cref{E:stab}.
For $n$ in the unipotent part of $P_{(j)}$, then
$_{k,j-r,t-j-k+r}nw_{k,j-r,t-j-k+r}^{-1}\in \iota (N_t,N_t)$ which again implies the invariance by \Cref{E:stab}.

Next, we discuss the required equivariance properties by the Levi-components.
First, for $m_1=(a_1,b_1)\in H_j\times G(W_j), m_2=(a_2,b_2) \in H_{r-k}\times G(W_j)$ we observe that by \Cref{E:equi} and the equivariance properties of $f$
\[\chi(m_1m_2)|m_2|^{s}\delta_{P(Y)}^{{\frac{1}{2}}}(m_1)|m_1|^{-s}\delta_{P(Y)}^{-{\frac{1}{2}}}(m_1)f(\phi(h)\iota(\iota_2(m_2)nm,\iota_1(m_1)g))=\]\[=f(\iota(\iota_2(\mathfrak{c}(a_2^t)\iota_3(b_2),\iota_1(a_1^{-1}))\phi(h)\iota(\iota_2(a_2),\iota_1(a_1)\iota_3(b_1)))\iota(nm,g))=\]\[=f(\iota(\iota_2(\mathfrak{c}(a_2^t),\iota_3(b_2^{-1})\iota_1(a_1^{-1})))\phi(h)\iota(\iota_2(a_2),\iota_1(a_1)\iota_3(b_1))\iota(nm,g))=\]\[=f(\phi(m_2^{-1}hm_1)\iota(nm,g)),\]
where we used for the second equality \Cref{E:stab}.
Therefore \[f(\phi(h)\iota(\iota_2(m_1)nm,\iota_1(m_1)g))=\]\[=\chi(m_1m_2)|m_2|^{s}\delta_{P(Y)}^{{\frac{1}{2}}}(m_2)|m_1|^{-s}\delta_{P(Y)}^{-{\frac{1}{2}}}(m_2)f(\phi(m_1^{-1}hm_2)\iota(nm,g)).\]
Finally, \[\int_{N_{k,j}\bs N}f(\phi(h)\iota(n\iota_2(m_2)m,g))\, \mathrm{d}n=\lvert m_2\lvert^{n-j-k-\eta}\int_{N_{k,j}\bs N}f(\phi(h)\iota(\iota_2(m_2)nm,g))\, \mathrm{d}n\]
by a change of variables and the explicit form of $N_{k,j}\bs N$ we give in the proof of \Cref{L:stabilizerofunipotent}. The required invariance thus follows.
Furthermore, for $m'$ an element in $H_k\times H_{r-k}\times H_{j-r}\times G(W_j)$ of the form $m'=(m_1',1_{r-k}, m_2', 1_{W_j})\in H_k\times 1_{r-k}\times H_{j-r}\times 1_{W_j}$ and $\ain{t}{j-r+k}{j}$ \[w_{k,j-r,t-j-k+r}m'w_{k,j-r,t-j-k+r}^{-1}\in \iota(H_t, H_t)\]  and therefore we have by \Cref{E:stab}
\[f(\phi(h)\iota(\iota_2(m')nm,g))=\chi(m_1'm_2')|m_1'm_2'|^{s}\delta_{P(Y)}^{{\frac{1}{2}}}(m_1'm_2')f(\phi(h)\iota(nm,g))\] and hence a change of variables shows that
\[A_{k,j}(f)(g,m'm,h)=\chi(m_1')\lvert m_1'\lvert^s\delta_{P(Y)}^{\frac{1}{2}}(m_1')\chi(m_2')\lvert m_2'\lvert^{s-r+k}\delta_{P(Y)}^{\frac{1}{2}}(m_2')A_{k,j}(f)(g,m,h).\]
Finally, it is not hard to see that $A_{k,j}'$ is an $H_r\times G(W_r)\times G(W)$-intertwiner, since the morphism respects right translation.

Therefore $A_{k,j}'$ is a well-defined morphism. Note that after normalizing both the Jacquet module and induction we obtain a morphism
\[A_{k,j}\colon \tau_{k,j}\ra (1\otimes\chi)\id_{Q_{(k,r-k)}\times P_{(j-r)}' \times P_{(j)}}^{H_r\times G(W_r)\times G(W)}(\underbrace{\chi\lvert-\lvert^{s+{\frac{k}{2}}}}_{H_k}\otimes 
    \]\[\otimes \sigma_{r-k,j}(\chi\lvert-\lvert^{-s-{\frac{r-k-j}{2}}}\otimes \chi\lvert-\lvert^{s+{\frac{j}{2}}})\otimes \underbrace{\chi\lvert-\lvert^{s+k+{\frac{j-r}{2}}}}_{H_{j-r}}\otimes S(G(W_j))).\]
Next we stratify $\Gamma_{k,j}$ by $\Gamma_{k,j,t},\,\ain{t}{j-r+k}{j}$, where 
\[\Gamma_{k,j,t}\coloneq \Omega_t\cap \Gamma_{k,j}=x_t w_{k,j-r,t+r-j-k}\iota(P)\] and set $l\coloneq j-t$.
Recall from \Cref{S:minguezfiltration} that \[(1\otimes\chi)\id_{Q_{(k,r-k)}\times P_{(j-r)}' \times P_{(j)}}^{H_r\times G(W_r)\times G(W)}(\underbrace{\chi\lvert-\lvert^{s+{\frac{k}{2}}}}_{H_k}\otimes \sigma_{r-k,j}(\chi\lvert-\lvert^{-s-j+{\frac{r-k}{2}}}\otimes \chi\lvert-\lvert^{s+{\frac{j}{2}}})\otimes \underbrace{\chi\lvert-\lvert^{s+k+{\frac{j-r}{2}}}}_{H_{j-r}}\otimes S(G(W_j)))\] has a filtration by the rank of the linear maps in $\sigma_{r-k,j},$ \textit{i.e.} by representations of the form
\[\tau_{k,j,l}'\coloneq (1\otimes\chi)\id_{Q_{(k,r-k)}\times P_{(j-r)} \times P_{(j)}}^{H_r\times G(W_r)\times G(W)}(\underbrace{\chi\lvert-\lvert^{s+{\frac{k}{2}}}}_{H_k}\otimes \omega_l(\chi\lvert-\lvert^{-s-j+{\frac{r-k}{2}}}\otimes \chi\lvert-\lvert^{s+{\frac{j}{2}}})\otimes\]
\[\otimes \underbrace{\chi\lvert-\lvert^{s+k+{\frac{j-r}{2}}}}_{H_{j-r}}\otimes S(G(W_j)))\]
for $\ain{l}{0}{r-k}.$
We will show that 
$A_{k,j}$ restricts to an isomorphism
$r_P(\scf^{s,\chi}({\Gamma_{k,j,t}}))\iso \tau_{k,j,j-t}'$. Note that by construction of $A_{k,j}$ and $\phi$, the $A_{k,j}$ restricts a priori to a morphism $r_P(\scf^{s,\chi}({\Gamma_{k,j,t}}))\ra \tau_{k,j,j-t}'$. To show that it is an isomorphism we use the Geometric Lemma.
Namely, we can compute $r_P(\scf^{s,\chi}({\Gamma_{k,j,t}})))$ by applying the Geometric Lemma to \[\sigma_t=\id_{H_t\ltimes N_t\times H_t\ltimes N_t\times \Delta G(W_t)}^{G(W)\times G(W)}(\chi\lvert-\lvert^{s+{\frac{t}{2}}}\otimes \chi\lvert-\lvert^{s+{\frac{t}{2}}}\otimes 1)\] and obtain isomorphisms 
\[A_{k,j,t}\colon r_P(\scf^{s,\chi}({\Gamma_{k,j,t}}))\iso F(w_{k,j-r,t-j-k+r})(\id_{H_t\times H_t\times \Delta G(W_t)}^{H_t\times G(W_t)\times G(W)}(\chi\lvert-\lvert^{s+{\frac{t}{2}}}\otimes \chi\lvert-\lvert^{s+{\frac{t}{2}}}\otimes 1))=\]
\begin{equation}\label{E:ter}=\id_{P'}^{H_r\times G(W_r)\times G(W)}\circ \id_{R}^{R'}(w_{k,j-,t-j-k+r}\circ\delta \circ r_{Q}(\lvert-\lvert^{s+{\frac{t}{2}}}\otimes \lvert-\lvert^{s+{\frac{t}{2}}}\otimes 1)),\end{equation}
where \[P'=Q_{(k,l,r-k-l)}\times P_{(j-r)}'\times P_{(j-l,l)},\, R\coloneq \Delta H_l\times \Delta G(W_j),\]\[ R'\coloneq H_l\times H_l\times G(W_j)\times G(W_j),\, Q\coloneq Q_{(k,j-r,t-j-k+r)}\times H_{j-l}\times \Delta H_l\times \Delta G(W_{j}).\]
Plugging in the definitions yields that above representation is 
\[\pi_{k,j,t}\coloneq \id_{Q'_{(k,r-k-l,l)}\times P_{(j-r)}'\times P_{(t,l)}}^{H_r\times G(W_r)\times G(W)}(\underbrace{\chi\lvert-\lvert^{s+{\frac{k}{2}}}}_{H_k}\otimes \underbrace{\chi\lvert-\lvert^{-s-j+{\frac{l+r-k}{2}}}}_{H_{r-k-l}}\otimes  \underbrace{\chi\lvert-\lvert^{s+{\frac{t}{2}}}}_{H_{t}} \otimes \]\[\otimes S(H_l)(\lvert-\lvert^{-s-j+{\frac{l-k}{2}}}\otimes \lvert-\lvert^{s+j-{\frac{l-k}{2}}})\otimes \underbrace{\chi\lvert-\lvert^{s+k+{\frac{j-r}{2}}}}_{H_{j-r}} \otimes S(G(W_j))).\]
Recall from our discussion of the Geometric Lemma we saw that $A_{k,j,t}$ is precisely 
\[A_{k,j,t}(f)=\underbrace{g}_{G(W)}\times \underbrace{m}_{H_r\times G(W_r)} \mapsto  \int_{N_{k,j,t}\bs N}p(f(\delta_t\iota(w_{k,j-r,r-k-l},1_W)\iota(nm,g)))\, \mathrm{d}n\,\]
where \[N_{k,j,t}=N_{k,j}\] by \Cref{L:stabilizerofunipotent} and $p$ is the projection to \[r_{H'}(\chi\lvert-\lvert^{s+{\frac{t}{2}}}\otimes \id_{P_{(t)}\times \Delta G(W_t)}^{G(W_t)\times G(W)}(\chi\lvert-\lvert^{s+{\frac{t}{2}}}\otimes 1))\] with 
\[H'=(N,1_W)\cap (w_{k,j-r,t+r-j-k}^{-1},1_W)H_t\times H_t\times \Delta G(W_t)(w_{k,j-r,t+r-j-k},1_W)).\]
In this case, this means we just forget the $H'$-action. 
Thus 
\[A_{k,j,t}(f)({g, m})=  \int_{N_{k,j}\bs N}f(\delta_t\iota(w_{k,j-r,r-k-l},1_W)\iota(nm,g))\, \mathrm{d}n=A_{k,j}(f)(g, m)(1_l').\]
Comparing the two group-actions on each side of the equation, we obtain a commutative diagram of the following form
\[
\begin{tikzcd}
    r_P(\scf^{s,\chi}({\Gamma_{k,j,t}}))\arrow[d, "A_{k,j}"]\arrow[r, "A_{k,j,t}"]&\pi_{k,j,t}\\
    \tau_{k,j,l}'\arrow[ru, "B_l"]&
\end{tikzcd}
\]
where $B_l\colon\tau_{k,j,l}'\iso \pi_{k,j,t}$ is the following parabolically induced isomorphism. Namely, \[B_l\coloneq \id_{Q{(k,r-k)}\times G(W_r)\times P_{(j)}}^{H_r\times G(W_r)\times G(W)}(\underbrace{B}_{H_{r-k}\times H_j}\otimes \underbrace{1}_{H_k\times G(W_r)\times G(W_j)}),\] and $B$ is the composition \[\omega_l(\chi\lvert-\lvert^{-s-j+{\frac{r-k}{2}}}\otimes \chi\lvert-\lvert^{s+{\frac{j}{2}}})\iso \]\[\iso \id_{\overline{Q_{r-k-l,l}}\times Q_{j-l,l}}^{H_{r-k}\times H_j}(\lvert-\lvert^{\frac{l}{2}} \otimes \lvert-\lvert^{-{\frac{l}{2}}}\otimes (\lvert-\lvert^{\frac{l-r+k}{2}}\otimes \lvert-\lvert^{\frac{j-l}{2}})S(H_l) )(\chi\lvert-\lvert^{-s-j+{\frac{r-k}{2}}}\otimes \chi\lvert-\lvert^{s+{\frac{j}{2}}})\iso \]
\[\iso \id_{\overline{Q_{r-k-l,l}}\times Q_{j-l,l}}^{H_{r-k}\times H_j}(\lvert-\lvert^{-s-j+{\frac{l+r-k}{2}}} \otimes \lvert-\lvert^{s+{\frac{j-l}{2}}}\otimes S(H_l) ).\]
The last isomorphism in this composition is obtained by sending a function $f'$ on $S(H_l)$ to $\chi\lvert-\lvert^{-s-j+{\frac{l}{2}}}\cdot f'$ and then parabolically inducing this morphism to $H_{r-k}\times H_{j}$.
Therefore $ A_{k,j}(f)$ induces an isomorphism from to $r_P(\scf_c^{s,\chi}(\Gamma_{k,j,t}))\ra \tau_{k,j,l}'$  and
 the 5-Lemma shows then that $A_{k,j}$ is an isomorphism.
\end{proof}
    As a corollary of the proof, we obtain the following.
    Let $\Omega^t= \bigcup_{i=0}^t\Omega_t$, $\Gamma_{k,j,t}\coloneq \Omega_t\cap \Gamma_{k,j}$
    and $\Gamma_{k,j}^t\coloneq \Omega^t\cap \Gamma_{k,j}$. Moreover, set
 \[\tau_{k,j,t}\coloneq r_{P}(\scf(\Gamma_{k,j,t})),\,\tau_{k,j}^t\coloneq r_{P}(\scf(\Gamma_{k,j}^t)).\] 
    Furthermore,
        $\tau_{k,j}$ has a filtration 
        by \[S_{k,j}^l\coloneq (1\otimes\chi)\id_{Q_{(k,r-k)}\times P_{(j-r)}' \times P_{(j)}}^{H_r\times G(W_r)\times G(W)}(\underbrace{\chi\lvert-\lvert^{s+{\frac{k}{2}}}}_{H_k}\otimes 
    \]\[\otimes S_l(\chi\lvert-\lvert^{-s-j+{\frac{r-k}{2}}}\otimes \chi\lvert-\lvert^{s+{\frac{j}{2}}})\otimes \underbrace{\chi\lvert-\lvert^{s+k+{\frac{j-r}{2}}}}_{H_{j-r}}\otimes S(G(W_j)))\]
       coming from the filtration of $\sigma_{r-k,j}$ by $S_l$, $\ain{l}{0}{r-k}$.
    \begin{corollary}\label{C:compatiblefiltrations}
         Then
        $A_{k,j}$ restricts to an isomorphism
        \[A_{k,j}\colon \tau_{k,j}^t\iso S_{k,j}^{j-t}.\]
    \end{corollary}
\section{Behavior on a boundary component}
Let ${\xi}$ be a character of the form \[{{\xi}}\coloneq\chi\lvert-\lvert ^{s-{\frac{1}{2}}},\, s\in \C,\, \chi^2=1\] and set for $k\in\NN$, ${\rho_k\coloneq \Z([1,k]_{\xi})=\chi\lvert\det \lvert^{s+{\frac{k}{2}}}}$.
More generally, for
$\ain{m}{0}{q_W}$, $\alpha=(\alpha_1,\ldots,\alpha_k)$ a partition of $m$,
define \[\rho_\alpha\coloneq \chi(\underbrace{\lvert-\lvert^{s+{\frac{\alpha_1}{2}}}}_{H_{\alpha_1}}\times\ldots\times \underbrace{\lvert-\lvert^{s+{\frac{\alpha_k}{2}}}}_{H_{\alpha_k}})=\Z([1,\alpha_1]_\xi)\times\ldots\times \Z([1,\alpha_k]_\xi) ,\] which is irreducible by \Cref{L:glirr2}. Define \[\sigma_t'\coloneq \id_{P_{(t)}\times P_{(t)}}^{G(W)\times G(W)}(\rho_t\otimes \rho_t\otimes S(G(W_t)))\] and
write ${\xi}_a\coloneq {\xi}\lvert-\lvert^a$ for $a\in\C$.

\fakesubsection{} 
In this subsection, we prove the following proposition.
\begin{prop}\label{L:boundary}
Let $\pi\in\ir(G(W))$.
    Then \[\dim_\C\ho_{G(W)\times G(W)}(\sigma_t',\pi\otimes\pi^\lor)\le 1.\]
\end{prop}
\begin{proof}
 We prove the claim by induction on $\dim_E W$, the case $t=0$ being \Cref{L:regularrep}.
Thus assume that $t\ge 1$ and let $d\coloneq -2s$ be such that ${\xi}_d^\lor\cong {\xi}_1.$
We also can assume without loss of generality that $\alpha_1\ge \ldots\ge\alpha_k$ and set $t\coloneq \alpha_1$.
Depending on $d$ we will differentiate several cases.

\textbf{Case 1: $\ain{d}{1}{t-1}$}

Note that in this case ${\xi}_t\ncong {\xi}_t^\lor$. Indeed, if ${\xi}_t^\lor\cong {\xi}_t$ we would obtain that $2t=d+1$ and therefore $2t<t+1$, which gives a contradiction. Observe moreover, that if there exists a non-zero morphism\[\sigma_t\ra\pi\otimes\pi^\lor,\] there exist by \Cref{L:irreduciblesubquotient} representations $\tau\in\ir(G(W_t))$ such that $\rho_t\rtimes\tau\sra \pi$. Since by \Cref{L:quotientglsegments} $\xi_t\times\rho_{t-1}\sra\rho_t$
we can find
by \Cref{L:quotientglsegments} $1\le r\in\ZZ_{>0}$ and $\delta\in\ir(G(W_r))$ such that \[({\xi}_t)^r\rtimes \delta\sra \pi\]
and $\delta$ is not a quotient of a representation of the form ${\xi}_t\rtimes\delta'$ for $\delta'\in \ir(G(W_{r+1}))$.
Using the MVW-involution and the fact that ${\xi}_t\ncong {\xi}_t^\lor$ we obtain that \[\pi\cong\delta_{{\xi}_t,r}\hra({\xi}_t^\lor)^r\rtimes \delta,\, \pi^\lor\cong\delta^\lor_{{\xi}_t,r}\hra({\xi}_t^\lor)^{r}\rtimes \delta^\lor\] by \Cref{L:Ganquotient}, where we also introduced this notation. 
    We now have by Frobenius reciprocity that
    \[\dim_\C\ho_{G(W)\times G(W)}(\sigma_t',\pi\otimes\pi^\lor)\le \dim_\C\ho_{G(W)\times G(W)}(\sigma_t',({\xi}_t^\lor)^r\rtimes \delta\otimes({\xi}_t^\lor)^{r}\rtimes \delta^\lor)=\]
    \[\dim_\C\ho_{H_r\times G(W_r)\times G(W)}(r_{P_{(r)}\times G(W)}(\sigma_t'),({\xi}_t^\lor)^r\otimes \delta\otimes({\xi}_t^\lor)^{r}\rtimes \delta^\lor)\]
    Next we apply the Geometric Lemma to $r_{P_{(r)}\times G(W)}(\sigma_t')$ and show that the only subquotient
    \[F(v_{k_1,k_2,k_3})(\rho_t\otimes \id_{G(W_r)\times P_{(t)}}^{G(W_r)\times G(W)}(\rho_t\otimes S(G(W_t)))\]
    admitting a morphism to $({\xi}_t^\lor)^r\otimes \delta\otimes\pi^\lor$ corresponds to
$k_1=0,\, k_2=r-1,\, k_3=1$. Here we use the representative $v_{k_1,k_2,k_3}$ chosen in \Cref{E:representative}.
    Indeed, if there would exist a non-zero morphism 
    \[F(v_{k_1,k_2,k_3})(\rho\otimes \id_{G(W_m)\times P_{(m)}}^{G(W_m)\times G(W)}( \rho\otimes S(G(W_m)))\ra ({\xi}_t^\lor)^r\otimes \delta\otimes({\xi}_t^\lor)^{r}\rtimes \delta^\lor),\]
    we would obtain a morphism
    $\Z([1,k_1]_{\xi})\times \rho'\times \Z([k_1+k_2+1,t]_{\xi})^\lor \ra ({\xi}_t^r)^\lor$
    and a morphism $\Z([k_1+1,k_1+k_2]_{\xi})\rtimes \rho''\ra \delta$ for suitable representations $\rho'\in \ir(H_{r-k_1-k_3})$ and $\rho''\in \ir(G(W_{k_2}))$.
    From the first morphism we obtain that $k_1+k_2$ is either $t-1 $ or $t$. From the second morphism we obtain that $k_1+k_2$ cannot be $t$, since otherwise, we would have a surjective morphism
    \[{\xi}_t\times \Z([k_1+1,t-1]_{\xi})\rtimes \rho''\sra\Z([k_1+1,k_1+k_2]_{\xi})\rtimes \rho''\sra \delta\] and hence we obtain a contradiction to the assumption on $\delta$ by the MVW-involution. Moreover, $k_1\le 1$ with equality if and only if $d=t$, which we excluded.
We thus showed that 
\[\dim_\C\ho_{G(W)\times G(W)}(\sigma_t',({\xi}_t^r)^\lor\rtimes\delta\otimes({\xi}_t^\lor)^{r}\rtimes \delta^\lor)\le\]
\[\dim_\C\ho_{H_r\times G(W_r)\times G(W)}(\id_{Q_{(1,r-1)}\times P_{(t-1)}\times P_{(t,r-1)}}^{H_r\times G(W_r)\times G(W)}(S(H_{r-1})\otimes {\xi}_t^\lor\otimes \rho_{t-1}\otimes \rho_t\otimes S(G(W_{t+r-1})),\]\[({\xi}_t^\lor)^r\otimes \delta\otimes({\xi}_t^\lor)^{r}\rtimes \delta^\lor).\]
Applying \Cref{E:Bernstein} with respect to $Q_{(1,r-1)}$ yields that the dimension of the last space is equal to the dimension of
\[\ho_{H_{1}l\times H_{r-1}\times G(W_r)\times G(W)}({\xi}_t^\lor\otimes \id_{H_{r-1}\times P_{(t-1)}\times P_{(t,r-1)}}^{H_{r-1}\times H_1\times G(W_r)\times G(W)}(S(H_{r-1})\otimes \rho_{t-1}\otimes \rho\otimes S(G(W_{t+r-1})),\]\[r_{\overline{Q_{(1,r-1)}}}(({\xi}_t^\lor)^r)\otimes \delta\otimes({\xi}_t^\lor)^{r}\rtimes \delta^\lor).\]
We thus get a morphism ${\xi}_t^\lor\otimes S(H_{r-1})\ra r_{\overline{Q_{(1,r-1)}}}(({\xi}_t^\lor)^r)$, which by \Cref{L:squareirreducible} is up to a scalar unique and factors through the inclusion
\[ {\xi}_t^\lor\otimes S(H_{r-1})\sra  {\xi}_t^\lor\otimes ({\xi}_t^\lor)^{r-1} \hra  r_{\overline{Q_{1,r-1)}}}(({\xi}_t^\lor)^r).\]
Applying this to our Hom-space we obtain that the dimension is bounded by
\[\dim_\C\ho_{H_{r-1}\times G(W_r)\times G(W)}(\id_{H_{r-1}\times  P_{(t-1)}\times P_{(t,r-1)}}^{H_{r-1}\times G(W_r)\times G(W)}(S(H_{r-1})\otimes  \rho_{t-1}\otimes \rho\otimes S(G(W_{t+r-1}))),\]\[({\xi}_t^\lor)^{r-1}\otimes \delta\otimes({\xi}_t^\lor)^{r}\rtimes \delta^\lor).\]
Applying first \Cref{E:Bernstein} with respect to the parabolic subgroup $P(X_{t,t+r-1})$ contained in the second copy of $G(W)$, then \Cref{L:regularrep} and then again \Cref{E:Bernstein}  with respect to $P(X_{t,t+r-1})$, it follows that the dimension is equal to
\[\dim_\C\ho_{G(W_r)\times G(W)}(\id_{P_{(t-1)}\times P_{(t,r-1)}}^{G(W_r)\times G(W)}( \rho_{t-1}\otimes \rho_t\otimes \xi_t^{r-1}\otimes S(G(W_{t+r-1})),\delta\otimes({\xi}_t^\lor)^{r}\rtimes \delta^\lor).\]
Since $\xi_t^{r-1}\times \rho_t\cong \rho_t\times\xi_t^{r-1}$ by \Cref{L:glirr2}, the dimension of the last space is equal to
\[\dim_\C\ho_{G(W_r)\times G(W)}(\id_{P_{(t-1)}\times P_{(r-1,t)}}^{G(W_r)\times G(W)}( \rho_{t-1}\otimes \xi_t^{r-1}\otimes \rho_{t}\otimes  S(G(W_{t+r-1})),\delta\otimes({\xi}_t^\lor)^{r}\rtimes \delta^\lor)\le \]
\[\dim_\C\ho_{G(W_r)\times G(W)}(\id_{P_{(t-1)}\times P_{(r,t-1)}}^{G(W_r)\times G(W)}( \rho_{t-1}\otimes \xi_t^{r}\otimes \rho_{t-1}\otimes  S(G(W_{t+r-1})),\delta\otimes({\xi}_t^\lor)^{r}\rtimes \delta^\lor), \]
where we used for the second inequality again that $\xi_t\times\rho_{t-1}\sra \rho_t$. Applying \Cref{E:Bernstein}, we see that this is equal to
\[\dim_\C\ho_{H_{r}\times G(W_r)\times G(W_r)}(({\xi}_t^\lor)^r\otimes \id_{P_{(t-1)}\times P_{(t-1)}}^{G(W_r)\times G(W_r)}(\rho_{t-1}\otimes \rho_{t-1}\otimes S(G(W_{t+r-1})),\]\[ \delta\otimes r_{\overline{P_{(r)}}}(({\xi}_t^\lor)^{r}\rtimes \delta^\lor))=\]\[ \dim_\C\ho_{H_{r}\times G(W_r)\times G(W_r)}(({\xi}_t^\lor)^r\otimes \sigma_{t-1},\delta\otimes r_{\overline{P_{(r)}}}(({\xi}_t^\lor)^{r}\rtimes \delta^\lor)).\]
Since ${\xi}_t\ncong {\xi}_t^\lor$ we can
apply \Cref{L:Ganquotient}
to see that the dimension is equal to
\[ \dim_\C\ho_{H_{r}\times G(W_r)\times G(W_r)}(({\xi}_t^\lor)^r\otimes \sigma_{t-1},({\xi}_t^\lor)^{r}\otimes\delta\otimes  \delta^\lor).\]
The induction hypothesis on $\dim_E W$ shows then that this space is $1$-dimensional.

\textbf{Case 2: $d\notin\{1,\ldots, t\}$ or $t=d$:}

Note that if there exists a morphism
\[\sigma_t\ra\pi\otimes\pi^\lor\] then there exists by \Cref{L:irreduciblesubquotient}
$\tau'\in \ir(G(W_t))$ such that $\rho_t\times\tau'\sra\pi$. Therefore there exists $\rho\in \ir(H_{m})$, $\rho\coloneq\rho_{\alpha}$ for $\alpha=(\alpha_1,\ldots,\alpha_{k})$, $\max_i\alpha_i=t$ and $\tau\in\ir(G(W_{m}))$ with $\rho\rtimes\tau\sra \pi$ such that there exists no $\ain{b}{1}{t}$ and $\tau'\in\ir(G(W_{m+b}))$ with $\rho_b\rtimes \tau'\sra \tau$.
We write $\rho=\rho_t\times\rho'$.
Using the MVW-involution, we obtain that $\pi\hra\rho^\lor\rtimes\tau$. 

We have 
\[\dim_\C\ho_{G(W)\times G(W)}(\sigma_t',\rho^\lor \rtimes\tau\otimes\pi^\lor)=\]
\[\dim_\C\ho_{H_{t}\times G(W_t)\times G(W)}(\rho_t\otimes\id_{G(W_t)\times P_{(t)}}^{G(W_t)\times G(W)}(\rho_t\otimes S(G(W_t))),r_{\overline{P_{(t)}}}(\rho^\lor\rtimes\tau)\otimes\pi^\lor)=\]
\[\dim_\C\ho_{H_t\times G(W_t)\times G(W)}(\rho_t^\lor\otimes\id_{G(W_t)\times P_{(t)}}^{G(W_t)\times G(W)}(\rho_t\otimes S(G(W_t))),r_{{P_{(t)}}}(\rho^\lor\rtimes\tau)\otimes\pi^\lor).\]
We now apply the Geometric Lemma to 
$r_{{P_{(t)}}}(\rho\rtimes\tau)$ and see for which $k_1,k_2,k_3$ a morphism from the left side exists to the respective subquotient. We claim now that if $d\notin\{1,\ldots, t\}$ it is $k_1=t,\, k_2=m-t,\, k_3=0$ and if $t=d$ it is either 
$k_1=t,\, k_2=m-t,\, k_3=0$ or $k_1=0,\,k_2=m-t,\,k_3=t$. Assume
\[\dim_\C\ho_{H_t\times G(W_t)\times G(W)}(\rho_t\otimes\id_{G(W_t)\times P_{(t)}}^{G(W_t)\times G(W)}(\rho_t\otimes S(G(W_t))),F(v_{k_1,k_2,k_3})(\rho^\lor\otimes\tau)\otimes\pi^\lor)\neq 0.\]
Plugging in the definition of \[F(v_{k_1,k_2,k_3})(\rho^\lor\otimes\tau)=\id_{Q_{(k_1,t-k_1-k_2,k_3)}\times P_{(k_2)}}^{H_t\times G(W_t)}\circ v_{k_1,k_2,k_3}\circ r_{Q_{(k_1,k_2,k_3)}\times P_{(m-k_1-k_3)}}(\rho^\lor\otimes\tau)\]
and applying Frobenius reciprocity together with \Cref{L:glred}
we obtain $\rho_{k_3}\cong \rho_{k_3}^\lor$ and hence $k_3=d$. Thus $k_3=0$ if $t\neq d$ and if $k_3\neq 0$, $k_3=d$ and hence $k_1=0$ and $k_2=m-t$ in this case. 
Moreover, if $k_3=0$ then $k_2$ cannot be different from $m-t$ as it otherwise implies the existence of an irreducible representation $\tau'\in \ir(G(W_{m+t-k_1}))$ and a non-zero morphism
$\Z([k_1-t,-1]_{{\xi}^\lor})\otimes\tau'\hra r_{P_{(t-k_1)}}(\tau)$ and hence a morphism
$\Z([1,t-k_1]_{\xi})\otimes\tau'\hra r_{\overline{P_{(t-k_1)}}}(\tau)$. By \Cref{E:Bernstein} this contradicts the assumption on $\tau$.
Thus, we have shown the claim that if $d\neq t$ $k_1=t, k_2=m-t,k_3=0$ and if $d=t$ $k_1=t, k_2=m-t,k_3=0$ or $k_1=0, k_2=m-t,k_3=t$.

\textbf{Case 2.1. $d\notin\{1,\ldots, t\}$:}

We just showed \[\dim_\C\ho_{G(W)\times G(W)}(\sigma_t',\rho^\lor \rtimes\tau\otimes\pi^\lor)\le\]
\begin{equation}\label{E:goodcase}\dim_\C\ho_{H_t\times G(W_t)\times G(W)}(\rho_t^\lor\otimes\id_{G(W_t)\times P_{(t)}}^{G(W_t)\times G(W)}(\rho_t\otimes S(G(W_t))),r_{(t,m-t)}(\rho^\lor)\rtimes\tau\otimes\pi^\lor).\end{equation}
We note that the $\rho^\lor$ is irreducible and that $\mathrm{Jac}_{\rho_t^\lor}(\rho^\lor)$, \emph{cf.} end of \Cref{S:2.1}, is irreducible. Indeed, since $\rho^\lor$ is irreducibly induced, we can write it as $\rho^\lor\cong (\rho_t^\lor)^k\times \rho_{\beta}^\lor$, where $\max_i\beta_i<t$. Then the Geometric Lemma gives that $\mathrm{Jac}_{\rho_t^\lor}(\rho^\lor)\hra \mathrm{Jac}_{\rho_t^\lor}((\rho_t^\lor)^k)\times \rho_\beta^\lor$, which by  \Cref{L:squareirreducible} and \Cref{L:glirr2} is irreducible and equal to $\rho'=(\rho_t^\lor)^{k-1}\times \rho_\beta^\lor$.
 Thus every morphism in \Cref{E:goodcase} factors through a morphism in
\[\ho_{H_t\times G(W_t)\times G(W)}(\rho_t^\lor\otimes\id_{G(W_t)\times P_{(t)}}^{G(W_t)\times G(W)}(\rho_t\otimes S(G(W_t))),\rho_t^\lor\otimes\rho'^\lor\rtimes\tau\otimes\pi^\lor),\]
whose dimension is by \Cref{E:Bernstein} and \Cref{L:regularrep} equal to 
\[\dim_\C\ho_{G(W_t)\times G(W)}(\id_{G(W_t)\times P_{(t)}}^{G(W_t)\times G(W)}(\rho_t\otimes S(G(W_t))),\rho'^\lor\rtimes\tau\otimes\pi^\lor)=\]
\[=\dim_\C\ho_{G(W_t)\times H_t\times  G(W_t)}(\rho_t\otimes S(G(W_t)),\rho'^\lor\rtimes\tau\otimes r_{\overline{P_{(t)}}}(\pi^\lor))=\]
\[=\dim_\C\ho_{H_t\times G(W_t)}(\rho_t\otimes \rho'\rtimes\tau^\lor, r_{\overline{P_{(t)}}}(\pi^\lor))=\]
\[=\dim_\C\ho_{G(W)}(\rho_t\times\rho'\rtimes\tau^\lor ,\pi^\lor)=\dim_\C\ho_{G(W)}(\pi,\rho^\lor\rtimes\tau).\]
But on the other hand, for each morphism $\sigma_t'\ra\pi\otimes\pi^\lor$ and
each embedding $\pi\hra \rho^\lor\rtimes\tau$ we have a morphism
\[\sigma_t'\sra \pi\otimes\pi^\lor\hra \rho^\lor\rtimes\tau\otimes\pi^\lor,\]
and hence
\[\dim_\C\ho_{G(W)}(\pi,\rho^\lor\rtimes\tau)\ge \dim_\C\ho_{G(W)\times G(W)}(\sigma_t',\rho^\lor \rtimes\tau\otimes\pi^\lor)\ge \]\[\ge\dim_\C\ho_{G(W)\times G(W)}(\sigma_t',\pi\otimes\pi^\lor)\cdot\dim_\C\ho_{G(W)}(\pi,\rho^\lor\rtimes\tau).\]
 In particular, we have
$\ho_{G(W)\times G(W)}(\sigma_t',\pi\otimes\pi^\lor)\le 1.$

\textbf{Case 2.2: $t=d$:}

Note that this is equivalent to $\rho_t\cong \rho_t^\lor$. 

Since $t=d$
    \[\dim_\C\ho_{H_t\times G(W_t)\times G(W)}(\rho_t^\lor\otimes\id_{G(W_t)\times P_{(t)}}^{G(W_t)\times G(W)}(\rho_t\otimes S(G(W_t))),F(v_{0,m-d,d})(\rho^\lor\otimes\tau)\otimes\pi^\lor)= \]\[=\dim_\C\ho_{H_t\times G(W_t)\times G(W)}(\rho_t^\lor\otimes\id_{G(W_t)\times P_{(t)}}^{G(W_t)\times G(W)}(\rho_t\otimes S(G(W_t))),F(v_{d,m-d,0})(\rho^\lor\otimes\tau)\otimes\pi^\lor)=\]\[=\dim_\C\ho_{H_t\times G(W_t)\times G(W)}(\rho_t^\lor\otimes\id_{G(W_t)\times P_{(t)}}^{G(W_t)\times G(W)}(\rho_t\otimes S(G(W_t))),r_{(d,m-d)}(\rho^\lor)\rtimes \tau \otimes\pi^\lor).\] As in Case 2.1 we see that 
\[\dim_\C\ho_{H_t\times G(W_t)\times G(W)}(\rho_t^\lor\otimes\id_{G(W_t)\times P_{(t)}}^{G(W_t)\times G(W)}(\rho_t\otimes S(G(W_t))),r_{(d,m-d)}(\rho^\lor)\rtimes\tau\otimes\pi^\lor)=\]
    \[\dim_\C\ho_{G(W)}(\rho\rtimes\tau^\lor ,\pi^\lor)=\dim_\C\ho_{G(W)}(\pi,\rho^\lor\rtimes\tau).\]
Using the same arguments as in the beginning of Case 2 and Case 2.1 we obtain \begin{equation}\label{E:firstrestriction}\dim_\C\ho_{G(W)\times G(W)}(\sigma_t',\rho^\lor \rtimes\tau\otimes\pi^\lor)\le 2\dim_\C\ho_{G(W)}(\pi,\rho^\lor\rtimes\tau).\end{equation} 
In a completely analogous fashion, we can show that 
\[\dim_\C\ho_{G(W)\times G(W)}(\sigma_t',\rho^\lor \rtimes\tau\otimes\rho^\lor\rtimes\tau^\lor)\le 2\dim_\C\ho_{G(W)}(\rho\rtimes \tau^\lor,\rho^\lor\rtimes\tau^\lor).\]
On the other hand, for each morphism $\pi\otimes \pi^\lor\hra \rho^\lor \rtimes\tau\otimes\rho^\lor\rtimes\tau^\lor$ and morphism $\sigma_t'\sra \pi\otimes \pi'$ we obtain a map in $\dim_\C\ho_{G(W)\times G(W)}(\sigma_t',\rho^\lor \rtimes\tau\otimes\rho^\lor\rtimes\tau^\lor)$ and hence
\[2\dim_\C\ho_{G(W)}(\rho\rtimes \tau^\lor,\rho^\lor\rtimes\tau^\lor)\ge\]
\[\dim_\C\ho_{G(W)\times G(W)}(\sigma_t',\pi\otimes\pi^\lor)\cdot \dim_\C\ho_{G(W)}(\pi,\rho^\lor \rtimes\tau)^2.\]
Next we prove the following lemma.
\begin{lemma}
    \[\dim_\C\ho_{G(W)}(\rho\rtimes \tau^\lor,\rho^\lor\rtimes\tau^\lor)\le 2.\]
\end{lemma}
\begin{proof}
    We write $\rho=\rho_t^k\times \Tilde{\rho}$, where $\Tilde{\rho}=\rho_\beta$ and $\max_i \beta_i<t=d$.
    Applying Frobenius reciprocity, and the Geometric Lemma, to \[\ho_{G(W)}(\rho\rtimes \tau^\lor,\rho^\lor\rtimes\tau^\lor)\] we obtain as above that
    \[\dim_\C\ho_{G(W)}(\rho\rtimes \tau^\lor,\rho^\lor\rtimes\tau^\lor)\le 2\dim_\C\ho_{H_{kt}\times G(W_{kt})}(\rho_t^k\otimes \Tilde{\rho}\rtimes \tau^\lor,\rho_t^k\otimes \Tilde{\rho}^\lor\rtimes\tau^\lor).\] Thus it is enough to show that \[\dim_\C\ho_{ G(W_{kt})}(\Tilde{\rho}\rtimes \tau^\lor,\Tilde{\rho}^\lor\rtimes\tau^\lor)\le 1.\] But this follows by applying Frobenius reciprocity and using the Geometric Lemma completely analogously as in the beginning of Case 2 and 2.1.
\end{proof}
We thus proved that \[4\ge \dim_\C\ho_{G(W)\times G(W)}(\sigma_t',\pi\otimes\pi^\lor)\cdot \dim_\C\ho_{G(W)}(\pi,\rho^\lor \rtimes\tau)^2\] and hence either $ \dim_\C\ho_{G(W)\times G(W)}(\sigma_t',\pi\otimes\pi^\lor)\le 1$ or \begin{equation}\label{E:onedim}\dim_\C\ho_{G(W)}(\pi,\rho^\lor \rtimes\tau)=1.\end{equation} We assume from now on the second case and thus by \Cref{E:firstrestriction} we obtain that \begin{equation}\label{E:bound2}\dim_\C\ho_{G(W)\times G(W)}(\sigma_t',\pi\otimes\pi^\lor)\le 2.\end{equation}
Fix an irreducible subrepresentation $\sigma_1$ of $\jac(\pi)$. 
Firstly, by \Cref{E:Bernstein} and the MVW-involution, $\pi\hra\rho_t\rtimes \sigma_1^\lor$. Applying Frobenius reciprocity to the map $\rho\rtimes\tau\sra\pi\hra\rho_t\rtimes \sigma_1$ and using the Geometric Lemma, gives with the usual argument $\rho'\rtimes\tau\sra\sigma_1$. We thus obtain a map \[\rho\rtimes\tau\sra \rho_t\rtimes \sigma_1\sra\pi.\] 
If $\jac(\pi)$ admits a second subrepresentation $\sigma_1'$, \emph{i.e.} has a socle of length at least $2$, the exact same argument would give a map 
\[\rho\rtimes\tau\sra \rho_t\rtimes \sigma_1'\sra\pi\] and hence $\dim_\C\ho_{G(W)}(\rho\rtimes\tau,\pi)>1$, which contradicts our assumption of \Cref{E:onedim} by the MVW-involution. 

Now we distinguish two cases, namely $\pi\cong \rho_t\rtimes \sigma_1$ or $\pi$ is a proper quotient of $\rho_t\rtimes \sigma_1$. The latter is easy to deal with. By \Cref{E:firstrestriction} it suffices to construct a non-zero map 
$\sigma_t'\ra \rho^\lor\rtimes\tau\otimes \pi^\lor$ which has image not isomorphic to $\pi\otimes\pi^\lor$. This we can do as follows.
\[\sigma_t'\sra \rho_t\rtimes \sigma_1\otimes \rho_t\rtimes \sigma_1^\lor\sra \rho_t\rtimes \sigma_1\otimes  \pi^\lor\hra \rho^\lor\rtimes\tau\otimes \pi^\lor.\]
Thus we assume from now one that $\pi\cong \rho_t\rtimes \sigma_1$.
We first assume that $\jac(\pi)$ is not irreducible. It thus contains a non-semi-simple subrepresentation $\sigma$ of length $2$. Applying \Cref{E:Bernstein} we obtain a map $\rho_t\rtimes\sigma\sra\pi$ such that the composition $\pi\cong\rho_t\rtimes\sigma_1\hra \rho_t\rtimes\sigma\sra\pi$ is non-zero and hence a scalar since $\pi\cong \rho_t\rtimes\sigma_1$. Therefore \begin{equation}\label{E:splittting}\rho_t\rtimes\sigma\cong\pi\oplus\rho_t\rtimes\sigma_2,\end{equation} where $\sigma_2$ is the unique quotient of $\sigma$. Next we denote the image of the map \[{S(G(W_t))\ra(\sigma\mvw)^\lor\otimes\sigma\mvw}\] corresponding to the identity map $\sigma\mvw\ra\sigma\mvw$, \emph{cf.} \Cref{L:regularrep}, by $I$. 
\begin{lemma}
    The representation $I$ is of length $3$ and isomorphic to the kernel of the map $(\sigma\mvw)^\lor\otimes\sigma\mvw\sra\sigma_1\otimes\sigma_2^\lor$.
\end{lemma}
\begin{proof}
    It is straightforward to see that $I$ has to be contained in the kernel. Moreover $I$ admits $\sigma_1\otimes\sigma_1^\lor$ and $\sigma_2\otimes\sigma_2^\lor$ as a quotient. Indeed, the composition
    \[S(G(W_t))\ra(\sigma\mvw)^\lor\otimes\sigma\mvw\sra(\sigma\mvw)^\lor\otimes\sigma_2^\lor\] has image $\sigma_2\otimes\sigma_2^\lor$. We can argue similarly for $\sigma_1\otimes\sigma_1^\lor$. Thus if $I$ is not the kernel we obtain that it is isomorphic to $\sigma_1\otimes\sigma_1^\lor\oplus\sigma_2\otimes\sigma_2^\lor$, which cannot be a subrepresentation of $(\sigma\mvw)^\lor\otimes\sigma\mvw$, since it contradicts the assumption that $\sigma$ and hence $\sigma\mvw$ are not semi-simple.
\end{proof}
We thus have a surjective map \[\sigma_t'\sra \id_{P_{(t)}\times P_{(t)}}^{G(W)\times G(W)}(\rho_t\otimes \rho_t\otimes I)\cong \pi\otimes\pi^\lor\oplus \rho_t\rtimes\sigma_2\otimes\rho_t\rtimes\sigma_2^\lor\oplus\rho_t\rtimes\sigma_2\otimes\pi^\lor.\] The last isomorphism stems from \Cref{E:splittting}. In particular we have that $\pi\ncong\rho_t\rtimes\sigma_2$ by \Cref{E:bound2} and we constructed surjective maps
\begin{equation}\label{E:surjmaps2}\sigma_t'\sra \rho_t\rtimes\sigma_2\otimes\pi^\lor,\,\sigma_t'\sra \rho_t\rtimes(\sigma\mvw)^\lor\otimes\pi^\lor.\end{equation} 
Note that by \Cref{E:splittting}, $\rho_t\rtimes(\sigma\mvw)^\lor\cong \rho_t\rtimes\sigma$.
Now applying the Geometric Lemma to the inclusion $\rho_t\otimes\sigma\hra r_{\overline{P_{(t)}}}(\rho^\lor\rtimes\tau)$ implies that at least one of $\sigma$ or $\sigma_2$ is a subrepresentation of $\rho'^\lor\rtimes\tau$. We pick one which is and denote it by $\sigma'$. We can now construct a non-zero map
$\sigma_t'\ra \rho^\lor\rtimes\tau\otimes\pi^\lor$ which has an image not isomorphic to $\pi\otimes \pi^\lor$. Namely, by \Cref{E:surjmaps2} we have
\[\sigma_t'\sra\rho_t\rtimes\sigma'\otimes\pi^\lor\hra \rho^\lor\rtimes\tau\otimes\pi^\lor.
\]
and thus we are done by \Cref{E:firstrestriction}.

Finally, assume $\jac(\pi)$ is irreducible and hence isomorphic to $\sigma_1$. By \Cref{E:Bernstein} and \Cref{L:regularrep} we have
\[\dim_\C\ho_{G(W)\times G(W)}(\sigma_t',\pi\otimes\pi^\lor)=\dim_\C\ho_{G(W)}(\rho_t\rtimes \jac(\pi)^\lor,\pi^\lor)= 1.\]
\end{proof}
\section{Proof of Theorem \ref{T:mainconjecture}}
\fakesubsection{}
Let $\pi\in \mathrm{Irr}( G(W) )$ be an irreducible representation, $\chi$ a unitary character and $s\in \C$. In this section, we study the space
\[\ho_{ G(W) \times  G(W) }(I_{W,W}(\chi,s),\pi\otimes \chi\pi^\lor).\]
\begin{prop}[{\cite[§1]{KudlaRallis}}]\label{T:existence}
The above space is non-empty, \emph{i.e.}
\[\dim_\C \ho_{ G(W) \times  G(W) }(I_{W,W}(\chi,s),\pi\otimes  \chi\pi^\lor)\ge 1.\]
\end{prop}
We quickly recall the construction of a functional in this space.
For $v\in \pi, v^\lor\in \pi^\lor$ write the matrix coefficient $\phi(g)\coloneq v^\lor(\pi(g)v)$. For $\Psi\in I_{W,W}(\chi,s)$ we then define
\[\Z(s,\chi,\phi,\Psi)\coloneq \int_{ G(W) }\phi(g)\Psi(x_0\iota(g,1))dg.\] The integral $\Z(s,\chi,\phi,\Psi)$ converges for $\mathfrak{Re}\,s>>0$ and admits a meromorphic continuation to the whole complex plane. Moreover, it can be written as a rational function in $q^{-s}$ and the leading term of the Laurent polynomial of $\Z(s,\chi,\phi,\Psi)$ at $s=s_0$ defines then an element in 
\[{\ho_{ G(W) \times  G(W) }(I_{W,W}(\chi,s_0),\pi\otimes \pi^\lor)}.\]

In the same paper the authors deal with the case $\pi=\pi^\lor=1$ and $\pi$ not appearing on the boundary if $W$ is symplectic.
\begin{theorem}[{\cite[Theorem 1.1, Lemma 1.4]{KudlaRallis}}]\label{T:maintheoremtrivial}
Let $W$ be a symplectic vector space and $\pi\in \ir(G(W))$ either trivial or not appearing on the boundary. Then
\[\dim_\C \ho_{ G(W) \times  G(W) }(I_{W,W}(\chi,s),\pi\otimes \pi^\lor )= 1.\]
\end{theorem}
\fakesubsection{}
We will now generalize \Cref{T:maintheoremtrivial} to arbitrary representations. 
\begin{theorem}\label{T:maintheorem}
Let $W$ be now either a symplectic, orthogonal or unitary vector space over $E$ and $G(W)\subseteq \gl(W)$ the corresponding symmetry group.
Let $\pi$ an irreducible representation of $ G(W) $. Then \[\dim_\C \ho_{ G(W)\times G(W)}(I_{W,W}(\chi,s),\pi\otimes \chi\pi^\lor)=1.\]
\end{theorem}
\begin{rem}
    Note that the proof of the theorem would also follow through in the case $G(W)=\m(W)$ in a completely analogous way if one would have proven \Cref{T:howedualII} for metaplectic covers of general linear groups. This is the only reason why for the moment we cannot state this theorem in its full generality.
\end{rem}
\begin{proof}
In the light of \Cref{L:boundary} it would be enough to show that there exists a unique $\ain{t}{0}{q_W}$, depending on $\pi$, such that every non-zero morphism in \[\ho_{ G(W)\times G(W)}(I_{W,W}(\chi,s),\pi\otimes \chi\pi^\lor)\] does vanish on $I_{t-1}$ and does not vanish on $I_{t}$. Indeed, assuming this and given two non-zero morphisms \[f_1,f_2\colon I_{W,W}(\chi,s)\ra \pi\otimes \chi\pi^\lor,\] both would induce a non-zero morphism on \[f_1',f_2'\colon I_{t-1}\bs I_{t}\cong \sigma_t\ra \pi\otimes \chi\pi^\lor\] and hence there would exists by \Cref{L:boundary} $\lambda\in \C$ such that $f_1'=\lambda f_2'$. The morphism $f_1-\lambda f_2$ is then again an element of $\ho_{ G(W)\times G(W)}(I_{W,W}(\chi,s),\pi\otimes \chi\pi^\lor),$ which vanishes on $I_{t}$ and hence must be identically $0$, proving that the $\mathrm{Hom}$-space is $1$-dimensional.

To show that every non-zero morphism vanishes on $I_{t-1}$ and does not vanish on $I_{t}$ for some $t$ depending on $\pi$, we have to fix some notation.
 We set for $a\in \C,\, k\in\NN$ \[\xi_a\coloneq \chi\lvert-\lvert^{s+a-{\frac{1}{2}}},\, \rho_k\coloneq \Z([1,k]_{\xi_0})=\chi\lvert\det\lvert^{s+{\frac{k}{2}}}.\] Moreover, let $d\coloneq -2s\in \C$ be such that $\xi_d\cong \xi_1^\lor$ as in the proof of \Cref{L:boundary}.
If $\pi$ does not lie on the boundary of $I_{W,W}(\chi,s)$, the claim follows immediately. On the other hand, if it does, there exists $\ain{j}{1}{q_W}$ 
and some morphism $I_{W,W}(\chi,s)\ra \pi\otimes \chi\pi^\lor$ which vanishes on $I_{j-1}$ and not on $I_{j}$. It thus induces a morphism $\sigma_j\ra \pi\otimes \chi\pi^\lor$ and hence there exists by \Cref{L:irreduciblesubquotient} $\sigma\in \ir(G(W_j))$ such that $\rho_j\rtimes \sigma\sra \pi$.
In this case, let $\rho\in \ir(H_m),\, \tau\in \ir(G(W_m))$ be irreducible representations satisfying the following conditions.
\begin{enumerate}
    \item There exists $1\le k\in \NN,\, t_1,\ldots, t_k\in \{1,\ldots,q_W\}$ such that 
    \[\rho\cong \rho_{t_1}\times\ldots \times\rho_{t_k}\] and we set $t\coloneq \max_j t_j$.
    \item We can realize $\pi$ as a quotient $\rho\rtimes\tau\sra \pi$.
    \item There does not exist $\ain{t'}{1}{q_W}$ and $\tau'\in \ir(G(W_{m+t'}))$ such that $\rho_{t'}\rtimes \tau'\sra \tau$.
    \item There does not exists $t'>t$ and a non-zero morphism $I_{W,W}(\chi,s)\ra \pi\otimes \chi\pi^\lor$ vanishing on $I_{t'-1}$.
\end{enumerate}
Note firstly that by \Cref{L:glirr2} a representation $\rho$ of this form is indeed irreducible and secondly such $\rho$ and $\sigma$ exist. Indeed, choose $t_1$ to be the maximal $j$ such that there exists a morphism $f\in \ho_{ G(W)\times G(W)}(I_{W,W}(\chi,s),\pi\otimes \chi\pi^\lor)$ such that $f$ vanishes on $I_{j-1}$ and does not vanish on $I_{j}$. The so obtained morphism $\sigma_{t_1}\ra\pi\otimes\chi\pi^\lor$ gives then $\sigma\in \ir(G(W_{t_1}))$ such that $\rho_{t_1}\rtimes \sigma\sra \pi$. 
 We then can write $\sigma$ as the quotient of $\rho_{t_2}\times\ldots \rho_{t_k}\rtimes \tau$ for a suitable $\tau$ which satisfies above requirement. Then $\pi$ is the quotient of $\rho\rtimes \tau$ of the desired form.
The maximality of $t$ implies that it suffices to show that every morphism $I_{W,W}(\chi,s)\ra \pi\otimes \chi\pi^\lor$ vanishes on $I_{t-1}$.

By the MVW-involution we can therefore realize $\pi$ as a subrepresentation of $\rho^\lor\rtimes\tau$ and hence we obtain a non-zero morphism
\[f\colon I_{W,W}(s,\chi)\ra \rho^\lor\rtimes\tau\otimes\chi\pi^\lor\]
for each morphism in $\ho_{ G(W)\times G(W)}(I_{W,W}(\chi,s),\pi\otimes \chi\pi^\lor)$.
Applying Frobenius reciprocity to this morphism we get a morphism
\[f'\colon r_{P_{(m)}\times G(W)}(I_{W,W}(s,\chi))\ra \rho^\lor\otimes\tau\otimes\chi\pi^\lor,\]
where \[r_{P_{(m)}\times G(W)}(I_{W,W}(s,\chi))\] admits a filtration with subquotients $\tau_{k,j},\, \ain{k}{0}{m},\, \ain{j}{m}{q_W}$ and $\tau_{k,j}$ has a filtration 
\[\tau_{k,j}^{j-r+k}\cong S_{k,j}^{r-k}\subseteq\ldots \subseteq\tau_{k,j}=\tau_{k,j}^{t}\cong S_{k,j}^{j-t}\subseteq\ldots\subseteq \tau_{k,j}^{j}\cong S_{k,j}^{0}\]
with subquotients $\tau_{k,j,t}=\tau_{k,j}^{t-1}\bs \tau_{k,j}^{t},\,\ain{t}{j-r+k}{j}$ by \Cref{T:filtrationreduction} and \Cref{C:compatiblefiltrations}.
Recall that \[(1\otimes\chi)\id_{Q_{(k,m-k)}\times P_{(j-m)}' \times P_{(j)}}^{H_m\times G(W_m)\times G(W)}({\chi\lvert-\lvert^{s+{\frac{k}{2}}}}\otimes 
    \]\[\otimes \sigma_{m-k,j}(\chi\lvert-\lvert^{-s-j+{\frac{m-k}{2}}}\otimes \chi\lvert-\lvert^{s+{\frac{j}{2}}})\otimes {\chi\lvert-\lvert^{s+k+{\frac{j-m}{2}}}}\otimes S(G(W_j)))\cong \tau_{k,j}\]
    and 
\[(1\otimes\chi)\id_{Q_{(k,m-k)}\times P_{(j-m)}' \times P_{(j)}}^{H_m\times G(W_m)\times G(W)}({\chi\lvert-\lvert^{s+{\frac{k}{2}}}}\otimes 
    \]\[\otimes \omega_{j-t}(\chi\lvert-\lvert^{-s-j+{\frac{m-k}{2}}}\otimes \chi\lvert-\lvert^{s+{\frac{j}{2}}})\otimes {\chi\lvert-\lvert^{s+k+{\frac{j-m}{2}}}}\otimes S(G(W_j)))\cong \tau_{k,j,t}.\]
Assume now that $f$ does not vanish on $I_{t-1}$.
 We have to differentiate three cases, depending on the value of $d$ and $t$.

\textbf{Case 1: $t\notin\{1,\ldots,d\}$:}
 
We will now show that $f'$ does not vanish on $\tau_{0,m}$ and there exists no non-zero morphism $\tau_{k,j}\ra \rho^\lor\otimes\tau\otimes\chi\pi^\lor$ for all other subquotients.
Assume there exists a non-zero morphism $\tau_{k,j}\ra \rho^\lor\otimes\tau\otimes\chi\pi^\lor$.
We apply first \Cref{E:Bernstein} with respect to $P_{(j)}$ to
\[(1\otimes\chi)\id_{Q_{(k,m-k)}\times P_{(j-m)}' \times P_{(j)}}^{H_m\times G(W_m)\times G(W)}({\chi\lvert-\lvert^{s+{\frac{k}{2}}}}\otimes 
    \]\[\otimes \sigma_{m-k,j}(\chi\lvert-\lvert^{-s-j+{\frac{m-k}{2}}}\otimes \chi\lvert-\lvert^{s+{\frac{j}{2}}})\otimes {\chi\lvert-\lvert^{s+k+{\frac{j-m}{2}}}}\otimes S(G(W_j)))\cong \tau_{k,j}\ra \rho^\lor\otimes\tau\otimes\chi\pi^\lor\]
and obtain a non-zero morphism
\[\id_{Q_{(k,m-k)}\times H_j}^{H_m\times H_j}(\underbrace{\chi\lvert-\lvert^{s+{\frac{k}{2}}}}_{H_k}\otimes 
     \sigma_{m-k,j}(\chi\lvert-\lvert^{-s-j+{\frac{m-k}{2}}}\otimes \chi\lvert-\lvert^{s+{\frac{j}{2}}}))\ra \rho^\lor\otimes \rho',\] for some suitable irreducible representation $\rho'$. 
If $k>0$, applying \Cref{E:Bernstein} with respect to $Q_{(k,m-k)}$, \Cref{L:glred} and the Geometric Lemma show that $\rho_k\cong \rho_k^\lor$ and hence $k=d$, which contradicts the assumption on $d$.
 Thus $k=0$. Moreover, if $j>m$, we obtain from \Cref{L:irreduciblesubquotient} an irreducible representation $\tau'$ such that $\rho_{j-m}\rtimes \tau'\sra \tau$ contradicting the assumption on $\tau$. Indeed, the parabolic subgroup $P_{(j-m)}'$ is conjugated to the standard parabolic subgroup $P_{(j-m)}$ and twisting an irreducible representation by an inner automorphism does not change its isomorphism class. Note that the exact same proof shows that if there exists a non-zero morphism $\tau_{k,j,t'}\ra \rho^\lor\otimes\tau\otimes\chi\pi^\lor$ for some $t'$ then $k=0$ and $j=m$. Thus if $f$ does not vanish on $I_{t-1}$, $f'$ does not vanish on $\tau_{0,m}^{t-1}$.
 
 Thus $f'$ restricts to a non-zero morphism on $\tau_{0,m}$, which is a subrepresentation of \[r_{P_{(m)}\times G(W)}(I_{W,W}(s,\chi))\] since $\Gamma_{0,m}=\Gamma^{0,m}$ is an open subset of $\L_W$, see the preamble of \Cref{T:filtrationreduction}. However
by \Cref{C:supportzeta2} and \Cref{C:compatiblefiltrations}, $f'$ vanishes on 
$S_{0,m}^{m-t+1}\cong \tau_{0,m}^{t-1},$ a contradiction.

\textbf{Case 2: $t=d$:}

Recall that this is equivalent to $\rho_t^\lor\cong \rho_t$.
We will now show that $f'$ does not vanish on $\tau_{0,m}$ or $\tau_{d,m}$ and there exists no non-zero morphism $\tau_{k,j}\ra \rho^\lor\otimes\tau\otimes\chi\pi^\lor$ for all other subquotients.
Assume there exists a non-zero morphism $\tau_{k,j}\ra \rho^\lor\otimes\tau\otimes\chi\pi^\lor$.
We apply first \Cref{E:Bernstein} with respect to $P_{(j)}$ to
\[(1\otimes\chi)\id_{Q_{(k,m-k)}\times P_{(j-m)}' \times P_{(j)}}^{H_m\times G(W_m)\times G(W)}({\chi\lvert-\lvert^{s+{\frac{k}{2}}}}\otimes 
    \]\[\otimes \sigma_{m-k,j}(\chi\lvert-\lvert^{-s-j+{\frac{m-k}{2}}}\otimes \chi\lvert-\lvert^{s+{\frac{j}{2}}})\otimes {\chi\lvert-\lvert^{s+k+{\frac{j-m}{2}}}}\otimes S(G(W_j)))\cong \tau_{k,j}\ra \rho^\lor\otimes\tau\otimes\chi\pi^\lor\]
and obtain a non-zero morphism
\[\id_{Q_{(k,m-k)}\times H_j}^{H_m\times H_j}(\underbrace{\chi\lvert-\lvert^{s+{\frac{k}{2}}}}_{H_k}\otimes 
     \sigma_{m-k,j}(\chi\lvert-\lvert^{-s-j+{\frac{m-k}{2}}}\otimes \chi\lvert-\lvert^{s+{\frac{j}{2}}}))\ra \rho^\lor\otimes \rho',\] for some suitable irreducible representation $\rho'$. 
Applying \Cref{E:Bernstein} with respect to $Q_{(k,m-k)}$, \Cref{L:glred} and the Geometric Lemma show that $\rho_k\cong \rho_k^\lor$ and hence $k=d$ or $k=0$.
 Moreover, if $j>m$, we obtain as in Case 1 an irreducible representation $\tau'$ such that $\rho_{j-m}\rtimes \tau'\sra \tau$ contradicting the assumption on $\tau$.
Observe that the exact same proof shows that if there exists a non-zero morphism $\tau_{k,j,t'}\ra \rho^\lor\otimes\tau\otimes\chi\pi^\lor$ for some $t'$ then $k=0$ or $k=d$ and $j=m$. Note that by the preamble to \Cref{T:filtrationreduction}
$\Gamma^{0,m}=\Gamma_{0,m}$ is an open subset of $\L_W$ and hence $\tau_{0,m}$ is a subrepresentation of $r_{P_{(m)}\times G(W)}(I_{W,W}(s,\chi))$.
If $f'$ does not vanish on $I_{t-1}$ it therefore induces a non-zero morphism on $\tau_{0,m}^{t-1}$ or $\tau_{d,m}^{t-1}$. But $\Gamma_{d,m}\cap \Omega^{d-1}=\emptyset$ so $\tau_{d,m}^{t-1}=0$ and therefore
$f'$ induces a non-zero morphism $\tau_{0,m}\ra \rho^\lor\otimes\tau\otimes\chi\pi^\lor$.
However, by \Cref{C:compatiblefiltrations} and \Cref{C:supportzeta2} this implies that $f'$ vanishes on $S_{0,m}^{m-t+1}\cong \tau_{0,m}^{t-1}$. Thus we also arrive in this case at a contradiction.

\textbf{Case 3: $\ain{d}{1}{t-1}$:}

We will first show that then $f$ does not vanish on $I_{d-1}$. Indeed, assume otherwise.
We set in this case for $d<b\in \NN$ $\rho'_b\coloneq \Z([d+1,b]_{\xi_0})=\chi\lvert-\lvert^{s+{\frac{b+d}{2}}}$.
Since $\rho'_t\times \rho_d\sra \rho_t$ by \Cref{L:quotientglsegments}, we can define $\rho'\in \ir(H_{m'})$ and $\delta\in \ir(G(W_{m'})$ as follows.
\begin{enumerate}
    \item There exists $1\le l\in \NN,\, d<b_1,\ldots, b_l\in \{1,\ldots,q_W\}$ such that 
    \[\rho'\cong \rho'_{b_1}\times\ldots\times \rho'_{b_l}.\]
    \item We can realize $\pi$ as a quotient $\rho'\rtimes\delta\sra \pi$ and set $b\coloneq\max_i b_i$.
    \item There does not exist $\ain{b'}{d}{q_W}$ and $\delta^\lor\in \ir(G(W_{m+b'}))$ such that $\rho'_{b'}\rtimes \delta^\lor\sra \delta$.
\end{enumerate}
Note that we can assume that $b\ge t$. Indeed, by \Cref{L:glirr2} we can assume without loss of generality that $t_1=t=\max_i t_i$ and hence there exist by \Cref{L:irreduciblesubquotient} $\tau'\in \ir(G(W_t))$ such that $\rho_t\rtimes\tau'\sra \pi$. Since $\rho_t'\times\rho_d\times\tau'\sra \rho_t\rtimes\tau'\sra \pi$, we obtain $\tau''\in\ir(G(W_{t-d}))$ such that $\rho_t'\rtimes\tau''\sra \pi$. Writing $\tau''$ as the quotient of $\rho'_{b_2}\times\ldots \rho'_{b_l}\rtimes \delta$ as desired shows that we can assume $b\ge t$.

By the MVW-involution we can therefore realize $\pi$ as a subrepresentation of $\rho'^\lor\rtimes\delta$ and hence we obtain a non-zero morphism
\[f\colon I_{W,W}(s,\chi)\ra \rho'^\lor\rtimes\delta\otimes\chi\pi^\lor.\]
Applying Frobenius reciprocity to this morphism we get a morphism
\[f''\colon r_{P_{(m')}\times G(W)}(I_{W,W}(s,\chi))\ra \rho'^\lor\otimes\delta\otimes\chi\pi^\lor,\]
where \[r_{P_{(m')}\times G(W)}(I_{W,W}(s,\chi))\] admits a filtration with subquotients $\tau_{k,j},\, \ain{k}{0}{m'},\, \ain{j}{m'}{q_W}$ by \Cref{T:filtrationreduction}.
As in the previous cases one sees that the only $\tau_{k,j}$ admitting morphisms to $\rho'^\lor\otimes\tau\otimes\chi\pi^\lor$ have to satisfy $k=0$. Moreover, if $j>m'+d$, we would obtain from \Cref{L:irreduciblesubquotient} a morphism $\rho_{j-m'}\rtimes \delta^\lor\sra \delta$ for a suitable $\delta^\lor$ and since $\rho'_{j-m'}\times \rho_d\sra\rho_{j-m'}$ by \Cref{L:quotientglsegments}, we would contradict the assumption on $\delta$. Thus $j\le m'+d$. 
For $j<m'+d$, a morphism $\tau_{0,j}\ra \rho'^\lor\otimes\tau\otimes\chi\pi^\lor$ does not vanish on $S_{0,j}^{m'}\cong \tau_{0,j}^{j-m'}$ by \Cref{C:supportzeta1} and hence does not vanish on $\tau_{0,j}^{d-1}$. By exactly the same argument we obtain that $f''$ has to vanish on all $\tau_{k,j,t'}$ with $k\neq 0$ or $j\neq m'+d$.

This implies that if $f$ vanishes on $I_{d-1}$, $f''$ restricts to a non-zero morphism on $\tau_{0,m'+d}$ and vanishes on all other $\tau_{k,j}$'s. But here we can again apply \Cref{C:compatiblefiltrations} and \Cref{C:supportzeta2} to see that then $f''$ vanishes on $S_{0,m'+d}^{m'-b+1}\cong \tau_{0,m'+d}^{b-1}$ and hence $f''$ vanishes on $ r_{P_{(m')}\times G(W)}(I_{b-1})$. Since we showed that $b\ge t$, $f$ vanishes on $I_{t-1}$. 

Therefore $f$ does not vanish on $I_{d-1}$. We now apply this restriction to $f'$. As in Case 1 and 2, we see that 
the subquotients $\tau_{k,j}$ of $r_{P_{(m)}\times G(W)}(I_{W,W}(s,\chi))$ and the subquotients $\tau_{k,j,t'}$ of $\tau_{j,k}$ admit a morphism to $\rho^\lor\otimes \tau\otimes\chi\pi^\lor$ only if $k=0$ and $j=m$ or $k=d$ and $j\ge m$.
Since for $k=d$, $\Gamma_{k,j}\cap \Omega^{d-1}=\emptyset$, we obtain that $f'$ must restrict to a non-zero morphism on $\tau_{0,m}^{d-1}\cong S_{0,m}^{m-d+1}$ and in particular it does not vanish on $\tau_{0,m}$, which, as we observed before, is a subrepresentation of $r_{P_{(m)}\times G(W)}(I_{W,W}(s,\chi))$. But $f'$ restricted to $\tau_{0,m}$ does vanish on $S_{0,m}^{m-t+1}\cong \tau_{0,m}^{t-1}$ by \Cref{C:compatiblefiltrations} and \Cref{C:supportzeta2}. Since  $\tau_{0,m}^{t-1}$ contains $\tau_{0,m}^{d-1}$ we arrive at a contradiction.
\end{proof}
\bibliographystyle{plain}
\bibliography{References.bib}
\end{document}